\documentclass[11pt]{article}
\usepackage{epsfig}
\usepackage{amsmath,amsthm,amssymb}

\usepackage{parskip}
\usepackage{color}
\usepackage{setspace}
\usepackage{mathtools}
\usepackage{dsfont} %\mathds{1}
\usepackage{bbm} %\mathbbm{E}
\usepackage{fullpage}
  \def\d{\delta} 
    \def\g{\gamma}
\def\G{\Gamma}  \def\k{\kappa}

\def\t{\tau}

\def\E{\mathbbm{E}}
\def\Pr{\mathbbm{P}}
\def\Ind{\mathbf{1}}
\def\whp{{\bf whp}}

\newcommand{\ignore}[1]{}
\newcommand{\Af}{{\mathcal I}_f}
\newcommand{\brac}[1]{\left(#1\right)}
\newcommand{\bfrac}[2]{\left(\frac{#1}{#2}\right)}
\newcommand{\eps}{\varepsilon}

\newtheorem{theorem}{Theorem}
\newtheorem{claim}{Claim}
\newtheorem{lemma}[theorem]{Lemma}
\newtheorem{proposition}[theorem]{Proposition}
\newtheorem{corollary}[theorem]{Corollary}

\def\PA{\text{PA}}
\def\depth{\text{depth}}

\begin{document}

\title{A phase transition in the evolution of bootstrap percolation processes on preferential attachment graphs}
 \author{Mohammed Amin Abdullah\thanks{m.a.abdullah@bham.ac.uk}  \footnotemark[3]
 \and Nikolaos Fountoulakis\thanks{ n.fountoulakis@bham.ac.uk} \thanks{School of Mathematics, University of Birmingham, Edgbaston, B15 2TT, U.K.. Research supported by the EPSRC Grant No. EP/K019749/1.} }
%\subjclass[2010]{Primary 60K30;60K35;05C90 Secondary 05C80;60C05}
%\address{School of Mathematics, University of Birmingham, Edgbaston, B15 2TT, U.K.} 
%\footnote{Research supported by the EPSRC Grant No. EP/K019749/1.}
%\email{m.a.abdullah@bham.ac.uk, n.fountoulakis@bham.ac.uk}

\maketitle

\begin{abstract}
The theme of this paper is the analysis of bootstrap percolation processes on random graphs 
generated by preferential attachment. This is a class of infection processes where vertices have two states: they are either~\emph{infected} or~\emph{susceptible}. 
At each round every susceptible vertex which has at least $r\geq 2$ infected neighbours becomes infected and remains so forever. 
Assume that initially $a(t)$ vertices are randomly infected, where $t$ is the total number of vertices of the graph. 
Suppose also that $r < m$, where $2m$ is the average degree.  
We determine a critical function $a_c(t)$ such that when 
$a(t) \gg a_c(t)$, complete infection occurs with high probability as $t \rightarrow \infty$, 
but when $a(t) \ll a_c (t)$, then with high probability the process evolves only for a bounded number of rounds and the final set of
infected vertices is asymptotically equal to $a(t)$. 
%The critical function satisfies $a_c(t)=o(t)$. 
%In other words, a sublinear initial infection leads to full infection. 
%In contrast, when $r > m$, we show deterministically that the final infected set has size at most $(m+1)a(t)$ and the above phenomenon does not occur regardless of $a(t)$. %That is, if $a(t) = o(t)$, then this is also the case for the final set. 
\end{abstract}

\section{Introduction} 

The dissemination of contagion within a network is a fundamental problem that arises in a wide spectrum 
of social and economic sciences. Among the mechanisms which underlie this phenomenon 
is a class of dissemination processes where local decisions (or \emph{microbehaviours}) aggregate into
a large outbreak or \emph{pandemic}. Quite frequently, these phenomena begin on a rather small scale and may end
up contaminating a large part of the network. What are the particular characteristics of a network that enable or 
inhibit such an outbreak? 

A general class of models that incorporates this kind of behaviour is what is called the \emph{general threshold model}~\cite{Klein2007}. 
Here it is assumed that each vertex has one of two states: it is either \emph{infected} or~\emph{susceptible}.
Furthermore, each vertex of the underlying graph is equipped with a threshold function which depends on the states of its neighbours. 
This function expresses the probability that this vertex remains in a particular state given the states of its neighbours.  
A central problem in viral marketing is given a network and such a set of functions, find a set of vertices $S$ which maximizes 
the expected number of infected vertices at the end of the process. 
In~\cite{KKT2003}, Kempe, Kleinberg and Tardos proved that finding such an optimal set is NP-hard. Moreover, they showed that 
it is NP-hard to approximate the size of the maximum expected outreach even within a polynomial factor.  See also~\cite{KKT05} 
for similar results. 

In this paper, we study an instance of this class of models known as \emph{bootstrap percolation processes}. 
This is a threshold model that was introduced in the context of mathematical physics by Chalupa, Leath and Reich~\cite{ChLeRe:79} 
in 1979 for magnetic disordered systems.
%  and has been re-discovered since then by several authors mainly due to its
% connections with various physical models.

A \emph{bootstrap percolation process} with \emph{activation threshold} an integer $r\geq 2$
on a (multi)graph $G=G(V,E)$ is a deterministic process.
Initially, there is a subset $\mathcal{I}_0=\mathcal{I}(0) \subseteq V$ of infected vertices, whereas every other vertex is susceptible.
This set can be selected either deterministically or randomly.
The process evolves in rounds, where in each round, if a susceptible vertex has at least $r$ edges connected to infected neighbours, 
then it also becomes 
infected and remains so forever. This is repeated until no more vertices become infected. We denote the final infected set by $\Af$.
We denote the set of susceptible (infected) vertices at round $\t$ in the process by $\mathcal{S}(\t)$ (respectively, $\mathcal{I}(\t)$).
Thus, $\mathcal{S}(\t)$, $\mathcal{I}(\t)$ form a partition of the vertex set $V$, and $\Af=\mathcal{I}(\infty)$. Of course, the above 
definition makes also perfect sense when $r=1$ -- in this case $\Af$ coincides with the set of vertices of the union of those
 components of $G$ which contain vertices in $\mathcal{I}_0$. 

Such processes (as well as several variations of them) have been  used as models
to describe several complex phenomena in diverse areas, from jamming transitions~\cite{tobifi06} and
magnetic systems~\cite{sadhsh02} to neuronal activity~\cite{Am-nn, ET09}.
Bootstrap percolation processes also have connections with the dynamics of the Ising model at zero
temperature~\cite{Fontes02},~\cite{GlauberMorris2009}. 
These processes have also been studied on a variety of graphs, such as trees~\cite{BPP06, FS08}, grids~\cite{CM02, holroyd03,
BBDM2010}, lattices on the hyperbolic plane~\cite{BootHyper2013},
hypercubes~\cite{BB06}, as well as on several distributions of random graphs~\cite{Am-bp, balpit07, ar:JLTV10}. 
A short survey regarding applications of bootstrap percolation processes can
be found in~\cite{AdL03}. The theme of this paper is the study of bootstrap percolation processes  on a random preferential attachment
random graph on 
$t$ vertices, which we denote by $\text{PA}_t(m,\d)$. 

\section{Preferential attachment graphs} 

The preferential attachment models have their origins in the work of Yule~\cite{Yule}, where a growing model is proposed in the context 
of the evolution of species. A similar model was proposed by Simon~\cite{Simon} in the statistics of language. 
The principle of these models was used by Barab\'asi and Albert~\cite{BarAlb} to describe a random graph model where vertices arrive 
one by one and each of them throws a number of half-edges to the existing graph. Each half-edge is connected to a vertex with probability
that is proportional to the degree of the latter. This model was defined rigorously by Bollob\'as, Riordan, Spencer and 
Tusn\'ady~\cite{DegSeq} (see also~\cite{Diam}). We will describe the most general form of the model which is essentially due to 
Dorogovtsev et al.~\cite{Dor} and Drinea et al.~\cite{Drinea}. Our description and notation below follows that 
from the book of van der Hofstad~\cite{Remco}.

The random graph $\text{PA}_t(m,\d)$ is parameterised by two constants: $m \in \mathbb{N}$, and $\d \in \mathbb{R}$, $\d > -m$.  
It gives rise to a random graph sequence (i.e., a sequence in which each member is a random graph), denoted by
$\brac{\text{PA}_t(m,\d)}_{t=1}^\infty$. The $t$th term of the sequence, $\text{PA}_t(m,\d)$ is a graph with $t$ vertices and $mt$
edges. Further, $\text{PA}_t(m,\d)$ is a subgraph of $\text{PA}_{t+1}(m,\d)$. We define $\text{PA}_t(1,\d)$ first, then use it to define
the general model $\text{PA}_t(m,\d)$ (the Barab\'asi-Albert model corresponds to the case $\d = 0$).

The random graph $\text{PA}_1(1,\d)$ consists of a single vertex with one self-loop. We denote the vertices of $\text{PA}_t(1,\d)$ by
$\{v_1^{(1)}, v_2^{(1)}, \ldots, v_t^{(1)}\}$. We denote the degree of vertex $v_i^{(1)}$ in $\text{PA}_t(1,\d)$ by $D_i(t)$. Then,
conditionally on $\text{PA}_t(1,\d)$, the growth rule to obtain $\text{PA}_{t+1}(1,\d)$ is as follows: We add a single vertex
$v_{t+1}^{(1)}$ having a single edge. The other end of the edge connects to $v_{t+1}^{(1)}$ itself with probability
$\frac{1+\d}{t(2+\d)+(1+\d)}$, and connects to a vertex $v_i^{(1)} \in \text{PA}_t (1,\d)$ with probability
$\frac{D_i(t)+\d}{t(2+\d)+(1+\d)}$ -- we write $v^{(1)}_{t+1}\rightarrow v_i^{(1)}$.  For any $t \in \mathbb{N}$, let $[t]=\{1,\ldots, t \}$. Thus,

\[ \Pr\brac{v^{(1)}_{t+1}\rightarrow v_i^{(1)} \mid \text{PA}_t(1,\d)} = \left\{ 
  \begin{array}{l l}
    \frac{1+\d}{t(2+\d)+(1+\d)} & \quad \text{for $i=t+1$,}\\
    \frac{D_i(t)+\d}{t(2+\d)+(1+\d)} & \quad \text{for $i \in [t]$}
  \end{array} \right.\]
The model $\text{PA}_t(m,\d)$, $m>1$, with vertices $\{ 1,\ldots, t \}$ is derived from
$\text{PA}_{mt}(1,\d/m)$ with vertices $\{v_1^{(1)}, v_2^{(1)}, \ldots, v_{mt}^{(1)}\}$ as follows: For each $i=1, 2, \ldots, t$, we 
 contract the vertices $\{v_{(i-1)+1}^{(1)}, v_{(i-1)+2}^{(1)}, \ldots, v_{(i-1)+t}^{(1)}\}$ into one super-vertex, and identify this
 super-vertex as $i$ in $\text{PA}_t(m,\d)$. When a contraction takes place, all loops and multiple edges are retained. Edges
shared between a set of contracted vertices become loops in the contracted super-vertex. Thus, $\text{PA}_t(m,\d)$ is a graph on $[t]$.

The above process gives a graph whose degree distribution follows a power law with exponent $3+ \delta /m$. This was suggested by 
the analyses in ~\cite{Dor} and~\cite{Drinea}. It was proved rigorously for integral $\delta$ by Buckley and 
Osthus~\cite{BuckleyOsthus2004}.
For a full proof for real $\delta$ see~\cite{Remco}. 
In particular, when $-m < \delta < 0$, the exponent is between 2 and 3. Experimental evidence has shown that this is the case for several
networks that emerge in applications (cf.~\cite{ar:StatMechs}). 
Furthermore, when $m\geq 2$, then $\text{PA}_t(m,\d)$ is \whp\ connected, but when 
$m=1$ this is not the case, giving rise to a logarithmic number of components (see~\cite{Remco}). 

We describe an alternative, though equivalent, direct construction of  $\brac{\text{PA}_t(m,\d)}_{t = 1}^{\infty}$. Let $\text{PA}_1(m,\d)$ be a single vertex with label $1$, having $m$ loops. Given $\text{PA}_{t-1}(m,\d)$, $t\geq 2$,  the construction of $\text{PA}_t(m,\d)$ is as follows: To add vertex $t$ to the graph, we split time step $t$ into $m$ sub-steps, adding one edge sequentially in each sub-step. For $j=1,2,\ldots,m$, denote the graph after the $j$th sub-step of time $t$ by $\text{PA}_{t,j}(m,\d)$. Hence $\text{PA}_t(m,\d) \equiv \text{PA}_{t,m}(m,\d)$. For notational convenience, let $\text{PA}_{t,0}(m,\d) = \text{PA}_{t-1}(m,\d)$.

Denote the $j$th edge added by $e_j$. One end of $e_j$ will be attached to vertex $t$ and the other end will be attached randomly to another vertex (which may be $t$). Let $g(t, j)$ be the random variable representing this vertex. For $j=1,2,\ldots,m$, let $D_i(t,j)$ be the degree of vertex $i$ in $\text{PA}_{t,j}(m,\d)$. That is, for $j=1,2,\ldots,m$, $D_i(t,j)$ the degree of vertex $i$ after both ends of $e_j$ have been attached. Furthermore, for notational convenience, let $D_{t}(t,0)=0$ and for $i \in [t-1]$, let $D_i(t,0)=D_i(t-1)$. 
 
Now, for $j=1,2,\ldots,m$, conditionally on $\text{PA}_{t,j-1}(m,\d)$, $\text{PA}_{t,j}(m,\d)$ is generated  according to the following probability rules: 

\[ \Pr\brac{g(t, j) = i \mid \text{PA}_{t,j-1}(m,\d)} = \left\{ 
  \begin{array}{l l}
    \frac{D_{t}(t, j-1)+1+j\d/m}{(2m+\d)(t-1)+2j-1+j\d/m} & \quad \text{for $i=t$,}\\
    \frac{D_i(t,j-1)+\d}{(2m+\d)(t-1)+2j-1+j\d/m} & \quad \text{for $i \in [t-1]$}
  \end{array} \right. .\]

It is not difficult to see that these two constructions give rise to the same probability distribution over realisations of $\brac{\text{PA}_t(m,\d)}_{t = 1}^{\infty}$. It will be sometimes convenient to refer to one form over the other.

\subsection{Results} 
Here as well as in the rest of the paper the term \emph{with high probability} (\whp) means with 
probability $1-o(1)$ in the space of $\text{PA}_t(m,\d)$, as $t\rightarrow \infty$.  
We will be using the same term for events over the product space between $\text{PA}_t(m,\d)$ and the choice of 
$\mathcal{I}_0$ on $[t]$.  
The selection of $\mathcal{I}_0$ is random and each vertex is infected initially with probability $p=p(t)=a(t)/t$, independently of any other vertex. Hence, if $t$ is large and $a(t) \rightarrow \infty$ as $t \rightarrow \infty$, the size of $\mathcal{I}_0$ is with high 
probability close to $a(t)$.

Let $X_t$ be a random variable on the above product space. 
If $a \in \mathbb{R}$, we write that $X_t \stackrel{p}{\rightarrow} a$ ($X_t$ \emph{converges to $a$ in probability}) if 
for any $\eps>0$ we have $\Pr\brac{|X_t - a| >\eps} \rightarrow 0$ as $t \rightarrow \infty$.  

Recently, Ebrahimi et al.~\cite{EGGS14} investigated a threshold phenomenon that occurs in the evolution of the process on a variant of the preferential
attachment model, that is very similar (though not identical) to $\text{PA}_t (m,\d )$. In our context, their results can be stated as follows. 
Let $\g=\frac{m}{2m+\d}$.
If $a(t) \gg t^{1-\gamma}\log t$, then \whp \ $\mathcal{I}_f= [t]$, that is, we have complete infection.
They also identified a subcritical range for $a(t)$. 
Assume first that $r \gamma \geq 1$. If $a(t) \ll t^{1-\gamma}$, then \whp\ $\mathcal{I}_f = \mathcal{I}_0$, that is no evolution occurs. 
Now, if $r\gamma < 1$, then the same holds but provided that $a(t) \ll t^{1-1/r}$. Since $\gamma < 1/r$, that is, 
$1-\gamma > 1 - 1/r$, it follows that this function is asymptotically smaller than the $t^{1-\gamma}$. 
Similar results were obtained by the two authors in~\cite{AbFount2014} for $\text{PA}_t (m,\d )$. 

In this paper, we complete the landscape and show that a critical phenomenon occurs ``around" the function 
$t^{1-\g}=:a_c (t) = a_c$. Let $\omega=\omega(t) \rightarrow \infty$  as $t \rightarrow \infty$ arbitrarily slowly.
Our results show that when $a(t) \gg a_c(t)$, there is complete infection \whp, but if
$a(t)\ll a_c(t)$ then either there is no evolution of the process or it halts in a bounded number of rounds. 
(In fact, for $r=2$ we show a slightly weaker result that requires $a(t)\leq a_c (t)/\log t$.)
In the latter case, 
the process accumulates only a small number of infections beyond those incurred initially, so that $\mathcal{I}_f$ is almost equal to
$\mathcal{I}_0$. 
Inside the critical window, that is, if $a(t)=\Theta(a_c(t))$, then with probability asymptotically bounded away from zero there is complete
infection, and with probability bounded away from zero we have similar behaviour as for the $a(t) \ll a_c(t)$ case.

The above can be formalized as follows. 
\\

\begin{theorem}[Supercritical case]\label{Supercritical case} 
If $r<m$ and $a(t)=\omega a_c(t)$ then all vertices in $\PA_t(m,\d)$ get infected \whp.
\end{theorem}

\vspace*{5mm}

\begin{theorem}[Subcritical case]\label{Subcritical case}
If $r \leq m$ then the following hold:
\begin{description}
\item[(i)] If $a(t)=a_c(t)/\omega$ and $r\g>1$, then \whp,  $\mathcal{I}_f=\mathcal{I}_0$.
\item[(ii)] If $a(t)=a_c(t)/\omega$ and $r \geq 3$ then 
$|\mathcal{I}_f|/|\mathcal{I}_0| \stackrel{p}{\rightarrow}1$ and \whp\ the process stops in at most 
$\lfloor \frac{1}{\g} \rfloor$ rounds.  
\item[(iii)] If $a(t)=a_c(t)/\log t$ and $r=2$, then $|\mathcal{I}_f|/|\mathcal{I}_0| \stackrel{p}{\rightarrow} 1$ and 
\whp\ the process stops in at most
$\lfloor \frac{1}{\g} \rfloor+1$ rounds.  
\end{description}
\end{theorem}

%\textbf{[Does the subcritical case hold for all $r \geq 2$? by a simple coupling argument it would, wouldn't it?]}

It should be noted that when $\d<0$, $r\g>1$ is always satisfied, since we insist that $r \geq 2$. 

\vspace*{5mm}

\begin{theorem}[Critical case]\label{Critical case}
Let $r \geq 3$ and $a(t)=\lambda a_c(t)$ where $\lambda$ is a constant.  
Then there exist $p_1 < p_2$ depending on $\lambda$ such that the following hold for any $t$ large enough:
\begin{description}
\item[(i)] if $r \leq m$, then the following holds with probability at least $p_1$:
 vertices are infected for at most 
$\lfloor \frac{1}{\g} \rfloor$ rounds, and $|\mathcal{I}_f|/|\mathcal{I}_0| < 1+\varepsilon$, for any $\varepsilon>0$.  
\item[(ii)] if $r<m$, then with probability at least $p_2$, there is a complete infection.
\end{description}
\end{theorem}
\vspace*{6mm}

The function $a_c(t)$ was also identified by the second author and Amini~\cite{ar:AmFount2012} in 
the case of inhomogeneous random graphs of rank 1. However, results of Amini~\cite{Am-bp} imply that if the kernel 
of such a random graph gives rise to a power law degree distribution with exponent larger than 3 (corresponds to 
$\delta >0$), then \whp, a \emph{sublinear} initial infection only results in a \emph{sublinear} outbreak. 
As our results and the results in~\cite{EGGS14} show this is not the case in the preferential attachment model. 
In other words, a sublinear initial infection leads to an outbreak where every vertex becomes infected, provided that 
the amount of the initial infection is not too small. Theorems~\ref{Supercritical case} and~\ref{Subcritical case} identify 
this critical amount. 

Lack of outbreak is also the case in random regular graphs of constant degree~\cite{balpit07} as well as in 
binomial random graphs with constant expected degree~\cite{ar:JLTV10}. In the latter case, the authors show that if $a(t) = o(t)$, then
$|\Af | / |\mathcal{I}_0 |\stackrel{p}{\rightarrow} 1$. This behaviour is radically different from that in the preferential attachment model,
 where Theorem~\ref{Supercritical case} implies that a sublinear initial infection may lead to pandemics.

\subsubsection{The cases $r = m$ and $r>m$}
It can be shown that there are a logarithmic number of self-loops in $\PA_t(m,\d)$.  For $r=m$, these loops make analysis
of the outcome difficult. This is a rather specific artefact of the model and, is not shared with slight variations of the model, 
e.g., one in which self-loops are not allowed. 

For $r>m$ the following ``folklore'' argument shows that if the number of initially infected vertices is sublinear, then
the final number will be sublinear as well: Let $G$ be the subgraph induced
by all the vertices in $\mathcal{I}_f$. The number of edges in $G$ is at least $(|\mathcal{I}_f|-|\mathcal{I}_0|)r$ but at the same time, 
the total number of edges in $G$ can be at most $m|\mathcal{I}_f|$. Therefore $(|\mathcal{I}_f|-|\mathcal{I}_0|)r\leq m|\mathcal{I}_f|$ 
implying $|\mathcal{I}_f| \leq \frac{r}{r-m}|\mathcal{I}_0|$. 

\subsection{Further notation and terminology}

Throughout this paper we let $\g=\g(m,\d)=\frac{1}{2+\d/m}$, hence $1-\g=\frac{1+\d/m}{2+\d/m}$. Observe the condition $\d>-m$ (which must be imposed), implies $0<\g<1$. Furthermore, $\d<0$ if and only if  $\frac{1}{2}<\g<1$.

For integers $i,j$ with $i \leq j$, we shall sometimes write $[i,j]$ to denote the set $\{i, i+1, \ldots, j\}$. We also use $S_i(t)$ to denote the sum of degrees for vertices in the interval $[1,i]$, i.e., $S_i(t)=\sum_{j=1}^iD_j(t)$.

We will sometimes say a vertex $j$ \emph{throws} an edge $e$ to vertex $i$ if, in the construction of $\text{PA}_j(m,\d)$, vertex $j$ connected edge $e$ to vertex $i$. We will also say $i$ \emph{receives} the edge $e$. 

%For two random variables $X$ and $Y$, $X \preceq Y$ denotes that $X$ is stochastically dominated by $Y$. 

Furthermore, for two non-negative functions $f(t), g(t)$ on $\mathbb{N}$ we write $f(t) \lesssim g(t)$ to denote that $f(t) = O(g(t))$. 
If, in addition, $g(t) = O(f(t))$, then we write $f(t) \asymp g(t)$. In this paper, the underlying asymptotic variable will
 always be $t$, the number of vertices in $\text{PA}_t(m,\d)$. 

We  use the notation $f(c) \overset{(m,\d)}\leq g(c)$ to mean that there is a constant $C(m, \d)$  such that $f(c) \leq C(m,\d) g(c)$,
 and $C(m,\d)$  depends only on $m, \d$. 

We will begin with some general results in the next section on the concentration of the degrees, which will be used mainly in the Proof of 
Theorem~\ref{Supercritical case}. 

\section{Vertex degrees: expectation and concentration}
As we mentioned, above the degrees in $\text{PA}_t(m, \d)$ roughly follow a power-law degree distribution with exponent
$3+\d/m$, that is,  the empirical probability mass function on the degrees scales like $\frac{1}{x^{3+\d/m}}$. 
In fact, many networks that emerge in applications have a degree distribution that follows a power law with exponent between $2$ and $3$ 
(cf.~\cite{ar:StatMechs} for example), which corresponds to $\d/m \in (-1,0)$. The Barabasi-Albert model gives power-law with
exponent $3$ ($\d =0$). Observe that the variance on the degrees is finite if and only if the exponent is greater than $3$ (corresponding to
$\d > 0$).

Consider two vertices $i$ and $j$; their total weight is $D_i(t)+D_j(t)+2\d$, meaning probability of an edge being thrown to them is proportional to this value. Now a vertex with degree $D_i(t)+D_j(t)$ would have weight $D_i(t)+D_j(t)+\d$. Thus, we cannot treat two separate vertices $i$ and $j$ as a single one of the combined degree, except when $\d=0$. In the special case that $\d=0$, the weight of a vertex is proportional to its degree, and the weight of a set of vertices is proportional to the sum of their degrees. When $\d=0$, we can treat a set of vertices as a bucket of \emph{half-edges}, or \emph{stubs}, conceptually distributing the stubs across the vertices however we like. However, when $\d \neq 0$, the weighting is non-linear. Conceptually grouping stubs together means that one has to sum their weights not their degrees. 

In summary, the probability of a vertex receiving the next edge thrown is proportional to its weight. The same holds for a set of vertices; the probability a set of vertices receiving an edge is proportional to the total weight of the set. When, and only when, $\d=0$, then the weight of a vertex is its degree, and the weight of a set is the total degree of the vertices in the set.

It is worth considering how $\d$ biases edge throws. Having $\d=0$ means edge throws are biased towards vertices in proportion to their degree. A negative $\d$ biases toward high degree vertices even more, since the proportional reduction in their weights is less. In fact, it is instructive to consider that if $m=1$ and $\d=-m$ (which this model does not permit), then the result would be that every vertex connects its single edge to the first vertex. 

Consider the case $\d>0$. This reduces the power of heavy vertices to attract edges. In fact, when $\d \gg m \gg 0$, the graph starts to looks fairly regular, since the $\d$ terms dominate in the update rules, and edges are thrown almost uniformly at random. 

A number of results on the degree sequence are collected in van der Hofstad \cite{Remco} which shows, amongst other things, that
$E[D_i(t)]=(1+o(1))a\bfrac{t}{i}^\g$ where $a$ is a constant that depends only on $m$ and $\d$.

\subsection{Sum of degrees}
We state the following without proof. It is a simple consequence of results in, e.g., \cite{Remco}.
\\

\begin{proposition}\label{CorToProp1}
There exist constants $C_\ell, C_u>0$ that depend only on $m$ and $\d$ such that for each vertex $i \in [t]$,
\begin{equation*}
C_\ell t^{\g}i^{1-\g} \leq \E[S_i(t)] \leq C_ut^{\g}i^{1-\g}.
\end{equation*}
\end{proposition}

We next derive a concentration results for the sum of degrees. Lemma \ref{ColLem1} is an elaboration of Lemma 2 in \cite{Colin}. 
Its proof, is in the appendix.
\\

\begin{lemma}\label{ColLem1}
Suppose $\d \geq 0$ and for a vertex $i \in [t]$, $i=i(t) \rightarrow \infty$. There exists a constant $K_0>0$ that depends only on $m$ and $\d$, such that the following holds for any constant $K>K_0$ and $h$ which is smaller than a constant that depends only on $m, \d$,
\begin{equation*}
\Pr\brac{S_i(t) < \frac{1}{K}\E[S_i(t)]}  \leq e^{-hi}
\end{equation*}
\end{lemma}
\vspace{3mm}

\begin{lemma}\label{SumConcLem}
Let $i \in [t]$, $i \geq 1$ be a vertex and let $\varepsilon >0$ be a constant.  If $\d<0$, then there exists a positive 
constant $c=c(m,\d,\varepsilon)$ that depends only on $m$, $\d$ and $\varepsilon$, such that with probability at least $1-e^{-ci}$,
\begin{equation}
 S_i(t) \geq (1-\varepsilon)\E[S_i(t)].    \label{9jhs65h4djkq}
\end{equation}
\end{lemma}

\begin{proof}
We will use a Doob martingale in conjunction with the Azuma-Hoeffding inequality. Define $M^{(m,\d)}_n(i,t)=\E[S_i(t)\mid \text{PA}_n(m,\d)]$. Observe, for $n=1,2,\ldots,i$, $M^{(m,\d)}_n(i,t)=\E[S_i(t)]$. 
Now we want to bound $|M^{(m,\d)}_{n+1}(i,t)-M^{(m,\d)}_n(i,t)|$ for $n \geq i$. Observe that $S_i(n)$ is measurable with respect to $\text{PA}_n(m,\d)$, and $\E[S_i(t)\mid S_i(n), \text{PA}_n(m,\d)]=\E[S_i(t)\mid S_i(n)]$, i.e., that the expectation of $S_i(t)$ is independent of $\text{PA}_n(m,\d)$ given $S_i(n)$. Hence, we will instead write $M^{(m,\d)}_n(i,t)=\E[S_i(t)\mid S_i(n)]$. We have, for $t>n$,
\begin{eqnarray*}
\E[S_i(t)+\d i\mid S_i(n)]&=&\E[\E[S_i(t)+\d i\mid S_i(t-1), S_i(n)]\mid S_i(n)]\\
&=&\E[\E[S_i(t)+\d i\mid S_i(t-1)]\mid S_i(n)].
\end{eqnarray*}
We will analyse the $m=1$ case first. Considering the inner conditional expectation,
\begin{eqnarray*}
\E[S_i(t)+\d i\mid S_i(t-1)] &=& S_i(t-1)+\d i + \frac{S_i(t-1)+\d i}{(2+\d)(t-1)+1+\d}\\
&=&\frac{(2+\d)t}{(2+\d)(t-1)+1+\d}\brac{S_i(t-1)+\d i}.
\end{eqnarray*}
Therefore,
\begin{eqnarray*}
\E[S_i(t)+\d i\mid S_i(n)]&=&\frac{t}{t-1+\frac{1+\d}{2+\d}}\E[S_i(t-1)+\d i\mid S_i(n)]\\
&=&\brac{S_i(n)+\d i}\prod_{k=n}^{t-1}\frac{k+1}{k+\frac{1+\d}{2+\d}}\\
&=&\brac{S_i(n)+\d i}\frac{\G(t+1)}{\G(t+\frac{1+\d}{2+\d})}\frac{\G(n+\frac{1+\d}{2+\d})}{\G(n+1)}.
\end{eqnarray*}

Consequently, 
\begin{eqnarray*}
&&\left|M^{(1,\d)}_{n+1}(i,t)-M^{(1,\d)}_n(i,t)\right|= \left|\E[S_i(t)\mid S_i(n+1)]-\E[S_i(t)\mid S_i(n)]\right|\\
&=& \frac{\G(t+1)}{\G(t+\frac{1+\d}{2+\d})}\left|\brac{S_i(n+1)+\d i}\frac{\G(n+1+\frac{1+\d}{2+\d})}{\G(n+2)}-\brac{S_i(n)+\d i}\frac{\G(n+\frac{1+\d}{2+\d})}{\G(n+1)}\right|\\
&=&\frac{\G(t+1)}{\G(t+\frac{1+\d}{2+\d})}\frac{\G(n+\frac{1+\d}{2+\d})}{\G(n+1)}\left|\brac{S_i(n+1)+\d i}\frac{n+\frac{1+\d}{2+\d}}{n+1} - \brac{S_i(n)+\d i}\right|.
\end{eqnarray*}

We have $\frac{n}{n+1} < \frac{n+\frac{1+\d}{2+\d}}{n+1} < 1$ and $S_i(n) \leq S_i(n+1) \leq S_i(n)+1$, so
\begin{eqnarray*}
\left|\brac{S_i(n+1)+\d i}\frac{n+\frac{1+\d}{2+\d}}{n+1} - \brac{S_i(n)+\d i}\right| &\leq& \brac{S_i(n)+\d i} \left|\frac{n+\frac{1+\d}{2+\d}}{n+1} - 1\right|+ \frac{n+\frac{1+\d}{2+\d}}{n+1}\\
&<&\frac{S_i(n)+\d i}{(2+\d)(n+1)}  + 1.
\end{eqnarray*}

Since $S_i(n)\leq 2i+n-i=n+i$ and $i \leq n$, the right-hand side is at most $2$:
\begin{equation*}
\frac{S_i(n)+\d i}{(2+\d)(n+1)} \leq \frac{n+i(1+\d)}{(2+\d)(n+1)} \leq \frac{n+n(1+\d)}{(2+\d)(n+1)}<1. 
\end{equation*}

Thus,
\begin{equation*}
\left|M^{(1,\d)}_{n+1}(i,t)-M^{(1,\d)}_n(i,t)\right| < 2 \frac{\G(t+1)}{\G(t+\frac{1+\d}{2+\d})}\frac{\G(n+\frac{1+\d}{2+\d})}{\G(n+1)}.
\end{equation*}

Recall that when $m \geq 1$ we define $\text{PA}_t(m,\d)$ in terms of $\text{PA}_{mt}(1,\d/m)$, and $S_a(b)$ in the former
corresponds to  $S_{ma}(mb)$ in the latter. Therefore, with $\g=\g(m,\d)=\frac{1}{2+\d/m}$,
\begin{eqnarray*}
\left|M^{(m,\d)}_{n+1}(i,t)-M^{(m,\d)}_n(i,t)\right|&=&\left|M^{(1,\d/m)}_{m(n+1)}(mi,mt)-M^{(1,\d/m)}_{mn}(mi,mt)\right|\\
&=& \left|\sum_{k=1}^{m}M^{(1,\d/m)}_{m(n+1)-k+1}(mi,mt)- M^{(1,\d/m)}_{m(n+1)-k}(mi,mt)\right|\\
&\leq&\sum_{k=1}^{m}\left|M^{(1,\d/m)}_{m(n+1)-k+1}(mi,mt)- M^{(1,\d/m)}_{m(n+1)-k}(mi,mt)\right|\\
&\leq&\frac{\G(mt+1)}{\G(mt+1-\g)}\sum_{k=1}^{m}\frac{\G(m(n+1)-k+1-\g)}{\G(m(n+1)-k+1)}.
\end{eqnarray*}

We have
\begin{eqnarray*}
\frac{\G(mn+k-\g)}{\G(mn+k)}&=&\frac{mn+k-1-\g}{mn+k-1}\frac{mn+k-2-\g}{mn+k-2} \ldots \frac{mn+1-\g}{mn+1}  \frac{\G(mn+1-\g)}{\G(mn+1)}\\
&\leq&\frac{\G(mn+1-\g)}{\G(mn+1)},
\end{eqnarray*}
so
\begin{equation*}
\sum_{k=1}^{m}\frac{\G(m(n+1)-k+1-\g)}{\G(m(n+1)-k+1)}=\sum_{k=1}^{m}\frac{\G(mn+k-\g)}{\G(mn+k)}\leq m \frac{\G(mn+1-\g)}{\G(mn+1)}.
\end{equation*}

Therefore,
\begin{equation*}
\left|M^{(m,\d)}_{n+1}(i,t)-M^{(m,\d)}_n(i,t)\right| \leq 2m\frac{\G(mt+1)}{\G(mt+1-\g)}\frac{\G(mn+1-\g)}{\G(mn+1)}.
\end{equation*}

Re-writing the above, we get 
\begin{eqnarray*}
\left|M^{(m,\d)}_{n+1}(i,t)-M^{(m,\d)}_n(i,t)\right|&\leq&  2m\frac{\G(mt+1-\g+\g)}{\G(mt+1-\g)}\frac{\G(mn+1-\g)}{\G(mn+1-\g+\g)}\\
&\leq& C_{m,\d}\bfrac{t}{n}^{\g},
\end{eqnarray*}
where $C_{m,\d}$ is a universal constant that depends only on $m$ and $\d$.

Now, applying the Hoeffding-Azuma inequality,
\begin{equation*}
\Pr\brac{S_i(t)-\E[S_i(t)]>d} \leq \exp\bfrac{-d^2}{C^2_{m,\d}\sum_{j=i+1}^{t}\bfrac{t}{j}^{2\g}}.
\end{equation*}

Since $\d<0$, we have $\sum_{j=i+1}^{t}\bfrac{t}{j}^{2\g}\leq K_1 t^{2\g}i^{1-2\g}$ for some  constant $K_1$.

Hence letting $d=\varepsilon\E[S_i(t)] \geq \varepsilon C_\ell t^{\g}i^{1-\g}$ for some constant $\varepsilon>0$,
\begin{equation*}
\Pr\brac{S_i(t)-\E[S_i(t)]>d} \leq \exp\bfrac{-\varepsilon^2C_\ell^2t^{2\g}i^{2(1-\g)}}{C^2_{m,\d}K_1t^{2\g}i^{1-2\g}} \leq e^{-ci}
\end{equation*}
for some constant $c=c(m,\d,\varepsilon)>0$ that depends only on $m$, $\d$ and $\varepsilon$.

\end{proof}

%========================================================================================================
%========================================================================================================
%========================================================================================================
%========================================================================================================

\section{Supercritical Case: Proof of Theorem~\ref{Supercritical case}}
The proof of this theorem relies on the fact that with high probability all of the early vertices of $\text{PA}_t(m,\d)$ become 
infected during the first round. Subsequently, the connectivity of the random graph is enough to spread the infection to the remaining 
vertices. The infection of the early vertices requires sufficiently high lower bounds on their degrees. We show these using the concentration 
results of the previous section together with a coupling with a P\'{o}lya urn process.

\subsection{P\'{o}lya Urns}\label{PolyaSection}

Consider the following P\'{o}lya urn process with red and black balls. Let $i \geq 2$ be an integer and let the weighting functions  for the red and black balls be  $W_R(k)=k+\d$ and $W_B(k)=k+(i-1)\d$, respectively. Under such a weighting scheme, if there are $a$ red balls and $b$ black balls, then the next time a ball is selected from the urn, the probability it is red is $\frac{W_R(a)}{W_R(a)+W_B(b)}=\frac{a+\d}{a+\d+b+(i-1)\d}=\frac{a+\d}{a+b+i\d}$. Whenever a ball is picked, it is placed back in the urn with another ball of the same colour. We can ask, if there are initially $a$ red and $b$ black balls, and we make $n$ selections, what is the probability that $d$ of those selections are red? 

To start with, one may calculate the probability of a particular sequence of $n$ outcomes. If  an $n$-sequence has $d$ reds followed by $n-d$ blues, then it has probability $p_{n,d,a,b}$ where 

\begin{align*}
p_{n,d,a,b} &= \frac{a+\d}{a+b+i\d}\frac{a+1+\d}{a+b+1+i\d}\ldots \frac{a+d-1+\d}{a+b+d-1+i\d}\\
&\qquad \times \frac{b+(i-1)\d}{a+b+d+i\d}\frac{b+1+(i-1)\d}{a+b+d+1+i\d}\ldots\frac{b+n-d-1+(i-1)\d}{a+b+n-1+i\d}\\
&=\frac{\G(a+d+\d)}{\G(a+\d)}\frac{\G(b+n-d+(i-1)\d)}{\G(b+(i-1)\d)}\frac{\G(a+b+i\d)}{\G(a+b+n+i\d)}.
\end{align*}

It is  not hard to see that this is the same probability for any $n$-sequence with $d$ reds and $n-d$ blues, regardless of ordering (this is the \emph{exchangeability} property of the P\'{o}lya urn process). As such, letting $X_R(n,a,b)$ be the number of reds picked when $n$ selections are made, we have 

\begin{equation}
\Pr(X_R(n,a,b)=d) = \binom{n}{d}p_{n,d,a,b}= \binom{n}{d}\frac{\G(a+d+\d)}{\G(a+\d)}\frac{\G(b+n-d+(i-1)\d)}{\G(b+(i-1)\d)}\frac{\G(a+b+i\d)}{\G(a+b+n+i\d)}. \label{polyaprob}
\end{equation}

Now let $i \geq 2$ and consider the vertices $[1,i]$ in $\brac{\text{PA}_t(m,\d)}_{t=i}^{\infty}$. With every vertex $t=i+1, i+2, \ldots$, there are $m$ edges  created, some of which may connect to vertices in $[1,i]$. We ask, what is the probability that an edge connects to $i$, given that it connects to some vertex in $[1,i]$? A coupling with the above  P\'{o}lya urn process is immediate: after the creation of $\text{PA}_i(m,\d)$, we create an urn with $D_i(i)$ red balls and $2mi-D_i(i)$ black balls. Every time a vertex $t>i$ connects an edge into the interval $[1,i]$, a selection is made in the urn process. A red ball is chosen if and only if the edge connects to $i$. 

To demonstrate that the probabilities correspond, suppose in $\text{PA}_{t,j-1}(m,\d)$ we have $D_i(t,j-1)=a$. Denoting $S_{i-1}(t,j-1)=\sum_{k=1}^{i-1}D_k(t,j-1)$, suppose also $S_{i-1}(t,j-1)=b$. Then it is easily checked that $\Pr\brac{g(t,j)=i \mid g(t,j) \in [1,i]}=\frac{a+\d}{a+b+i\d}$. Hence, if in $\text{PA}_t(m,\d)$ there are $n$ edges with one end in $[1,i]$ and the other end in $[i+1,t]$, then the probability that $d$ of those edges are attached to vertex $i$ is given by \ref{polyaprob}. As such, we have the following the proposition. 
\\

\begin{proposition}\label{PolyPAConnectProp}
Let $m\geq 1, i \geq 2$ be integers and let $\d>-m$ be a real. Suppose a P\'{o}lya urn process starts with $a \leq 2m$ red and $b=2mi-a$ black balls, and has weighting functions $W_R(k)=k+\d$ and $W_B(k)=k+(i-1)\d$ for the red and black balls, respectively. Let the random variable $X_R(n,a) = X_R(n,a,2mi-a)$ count the total number of red choices after $n$ selections have been made. Furthermore, consider a random graph $\text{PA}_t(m,\d)$. If $t \geq i$, then for $0 \leq d \leq n$,

\begin{equation*}
\Pr\brac{D_i(t)=d+a \mid S_i(t)-2mi=n, D_i(i)=a}=\Pr(X_R(n,a)=d).
\end{equation*}

\end{proposition}
The following lemma will be used to bound individual vertex degrees. 
\\

\begin{lemma} \label{polyaLemm}
Let $X_R(n,a)$ be the random variable defined in Proposition \ref{PolyPAConnectProp} and let $I=i(2m+\d)-1$. 
Then for $1 \leq d \leq n$,
\begin{equation}
\Pr(X_R(n,a)=d) \overset{(m,\d)}\leq   \frac{1}{d}\bfrac{Id}{I+n-d}^{a+\d}e^{-\frac{dI}{I+n}} ,  \label{idjgoosijh}
\end{equation}
and
\begin{equation}
\Pr(X_R(n,a)=0) \leq  \bfrac{I}{I+n}^{a+\d}. \label{h68h4gghjgkrweqw}
\end{equation}
\end{lemma}

\begin{proof}
As per Equation \eqref{polyaprob}, 
\[
\Pr(X_R(n,a)=d) = \binom{n}{d}\frac{\G(a+d+\d)}{\G(a+\d)}\frac{\G(b+n-d+(i-1)\d)}{\G(b+(i-1)\d)}\frac{\G(a+b+i\d)}{\G(a+b+n+i\d)}.
\]
That is, since $a+b =2mi$ and $a+b+i\d = i (2m + \d)$, we have 
\begin{equation*}
\Pr(X_R(n,a)=d)=\binom{n}{d}\frac{\G(a+\d+d)}{\G(a+\d)}\frac{\G(i(2m+\d)+n-(a+\d + d))}{\G(i(2m+\d)-(a+\d))}
\frac{\G(i(2m+\d))}{\G(i(2m+\d)+n)} %\label{fd4f8sg4}
\end{equation*}

We re-write the above as 
\begin{equation}
\Pr(X_R(n,a)=d)=\binom{n}{d}\frac{\G(a+\d+d)}{\G(a+\d)}\frac{\G(I+1+n-(a+\d + d))}{\G(I+1-(a+\d))}\frac{\G(I+1)}{\G(I+1+n)}.\label{jidjfog58748} 
\end{equation}

Suppose first that $d>0$. We can write the above as 
\begin{equation}
\Pr(X_R(n,a)=d)=\frac{\G(a+\d+d)}{d!\G(a+\d)}\frac{\G(I+1)}{\G(I+1-(a+\d))}\frac{(n)_d\G(I+1+n-(a+\d + d))}{\G(I+1+n)} \label{7j56fg}
\end{equation}
($(n)_d$  denotes the falling factorial $(n)_d=n(n-1)\ldots(n-d+1)$).

Now we bound \eqref{7j56fg}: using (essentially) Stirling's formula (\eqref{useful} in the appendix) observe  that $\G(a+\d+d) \overset{(m,\d)}\leq e^{-(a+\d+d-1)}(a+\d+d-1)^{a+\d+d-\frac{1}{2}}$. Furthermore, $d! \geq d^{d+\frac{1}{2}}e^{-d}$, so

\begin{eqnarray}
\frac{\G(a+\d+d)}{d!\G(a+\d)} &\overset{(m,\d)}\leq& \frac{e^{-(a+\d+d-1)}(a+\d+d-1)^{a+\d+d-\frac{1}{2}}}{d^{d+\frac{1}{2}}e^{-d}}\nonumber \\
&=&e^{-(a+\d-1)}(a+\d+d-1)^{a+\d-1}\bfrac{a+\d+d-1}{d}^{d+\frac{1}{2}}\nonumber \\
&\overset{(m,\d)}\leq & (a+\d+d-1)^{a+\d-1}\nonumber\\
 &\overset{(m,\d)}\leq & d^{a+\d-1}. \nonumber%\label{h5g4asf}
\end{eqnarray}

Also by \eqref{useful2}, $\frac{\G(I+1))}{\G(I+1-(a+\d))} \overset{(m,\d)}\leq I^{a+\d}$, and so  

\begin{equation}
\frac{\G(a+\d+d)}{d!\G(a+\d)}\frac{\G(I+1)}{\G(I+1-(a+\d))} \overset{(m,\d)}\leq  \frac{1}{d}(Id)^{a+\d}. \label{6484duju}
\end{equation}

Now,
\begin{equation*}
\frac{(n)_d\G(I+1+n-(a+\d + d))}{\G(I+1+n)} = \frac{n}{I+n}\frac{n-1}{I+n-1}\ldots\frac{n-(d-1)}{I+n-(d-1)}\frac{\G(I+1+n-(a+\d + d))}{\G(I+n-(d-1))}.
\end{equation*}

We have 
\[
\frac{n}{I+n}\frac{n-1}{I+n-1}\ldots\frac{n-(d-1)}{I+n-(d-1)} \leq \bfrac{n}{I+n}^d \leq e^{-\frac{dI}{I+n}}.
\]

Furthermore,
\begin{equation*}
\frac{\G(I+1+n-(a+\d+d))}{\G(I+n-(d-1))} \overset{(m,\d)}\leq\frac{1}{(I+n-d)^{a+\d}}.
\end{equation*}

Consequently, we have the following bound:

\begin{equation*}
\Pr(X_R(n,a)=d) \overset{(m,\d)}\leq \frac{1}{d} \bfrac{Id}{I+n-d}^{a+\d}e^{-\frac{dI}{I+n}}.
\end{equation*}

Now suppose $d=0$, then going back to \eqref{jidjfog58748} we have
\begin{equation*}
\Pr(X_R(n,a)=0)=\frac{\G(I+1)}{\G(I+1-(a+\d))}\frac{\G(I+1+n-(a+\d))}{\G(I+1+n)} \lesssim \bfrac{I}{I+n}^{a+\d}. 
\end{equation*}

\end{proof}

%========================================================================================================================
%========================================================================================================================
%========================================================================================================================

\begin{proof}[\textbf{Proof of Theorem \ref{Supercritical case}}]
For convenience, we rewrite as $a(t)=\omega^{10}a_c(t)$ where $\omega=\omega(t) \rightarrow \infty$ arbitrarily slowly (we can assume $\omega \leq \log t$, since if not, we can just substitute $\log t$ for it and get full infection \whp; a larger $\omega$ can only increase the probability of this happening). 

Let $\k=\lceil\omega^{1+\d/m}\rceil$ and choose $[\k]$ as a core. We wish to show all vertices in the core are infected for this $a(t)$.

For $\d \geq 0$,  we apply Lemma \ref{ColLem1}, setting $h=\frac{\log \k}{\k}$ , so that for some constant $K_\ell$, we have
$S_\k(t) \geq K_\ell t^\g \k^{1-\g}$ \whp. 
For $\d<0$, we apply Lemma \ref{SumConcLem} to get the same result. We set $n=n_\k(t)=K_\ell t^\g \k^{1-\g}-2m\k$. 

Now we wish to show that \whp, $D_i(t) \geq \bfrac{t}{\omega^{1+\d/m}}^{\g}\frac{1}{z}$ over all $i \in [\k]$, for some appropriately chosen $z=z(t) \rightarrow \infty$.
Applying Lemma \ref{polyaLemm}  with $I=\k(2m+\d)-1$,
\begin{eqnarray*}
\Pr\brac{X_R(n,a) \leq \frac{n}{\k z}} &=&\sum_{d=0}^{n/(\k z)}\Pr(X_R(n,a)=d)\\
&\lesssim& \bfrac{I}{I+n}^{a+\d}+ \sum_{d=1}^{n/(\k z)} \bfrac{I}{I+n-d}^{a+\d}d^{a+\d-1}e^{-\frac{dI}{I+n}}\\
&\leq & \bfrac{I}{I+n}^{a+\d}+  \frac{I^{a+\d}}{(I+n-n/(\k z))^{a+\d}}\sum_{d=0}^{n/(\k z)} d^{a+\d-1}.
\end{eqnarray*}

Since $\k \rightarrow \infty$ and $z \rightarrow \infty$ as $t \rightarrow \infty$, we have $n/(\k z)=o(n)$, so $\frac{1}{(I+n-n/(\k z))^{a+\d}}\lesssim \frac{1}{(I+n)^{a+\d}}$. 

Furthermore,
\[
\sum_{d=0}^{n/(\k z)} d^{a+\d-1} \lesssim \int_0^{n/(\k z)} x^{a+\d-1}  \,\mathrm{d}x \leq \frac{1}{a+\d}\bfrac{n}{\k z}^{a+\d}.
\]
Hence,
\begin{equation*}
\Pr\brac{X_R(n,a) \leq \frac{n}{\k z}} \lesssim \bfrac{I}{I+n}^{a+\d}+ \bfrac{I}{I+n}^{a+\d}\bfrac{n}{\k z}^{a+\d}\lesssim \bfrac{I}{I+n}^{a+\d}+ \frac{1}{z^{a+\d}} \leq 
\frac{1}{z^{m+\d }}.
\end{equation*}
We choose $z=\omega^2$. Then
\begin{equation*}
\bfrac{I}{I+n}^{a+\d}=\bfrac{\k(2m+\d-1)}{\k(2m+\d-1)+  K_\ell t^\g \k^{1-\g}-2m\k}^{a+\d}
\lesssim \bfrac{\omega^{1+\d/m}}{t}^{\g(a+\d)}
=o\bfrac{1}{z^{a+\d }}.
\end{equation*}

Thus, 
\begin{equation*}
\Pr\brac{X_R(n,a) \leq \frac{n}{\k z}} \lesssim \frac{1}{z^{m+\d }} \, .
\end{equation*}

Taking a union bound over all vertices in $[\k]$, we have a probability asymptotically bounded by $\bfrac{\omega}{z^m}^{1+\d/m}=o(1)$.

So given  $D_i(t) \geq \bfrac{t}{\omega^{1+\d/m}}^{\g}\frac{1}{\omega^2}$ for each $i \in [\k]$, we calculate the expectation of the number of infected neighbours a vertex in the core has. This would be at least
\[
\frac{a(t)}{2mt} \bfrac{t}{\omega^{1+\d/m}}^{\g}\frac{1}{\omega^2}= \frac{\omega^{8}}{2m}\bfrac{1}{\omega^{1+\d/m}}^{\g} \geq \omega^7
\]
for large enough $t$.

To calculate the probability that that at least $r$ neighbours are infected for a fixed vertex $i$ in the core, we bound the corresponding binomial random variable. 
Suppose $N=N(t)\rightarrow \infty$, $p=p(t)\rightarrow 0$ but  $Np \rightarrow \infty$. Then for large enough $t$, $\Pr(\text{Bin}(N,p) < r) \leq e^{-Np/2}$.   

Therefore, 
\[
\Pr\brac{\text{Bin}\brac{D_i(t),\frac{a(t)}{t}} < r \mid D_i(t) \geq \bfrac{t}{\omega^{1+\d/m}}^{\g}\frac{1}{\omega^2}} \leq e^{-\omega^7/2}
\]
and so the probability that any of the core vertices fail to be infected is at most $\omega^{1+\d/m}e^{-\omega^7/2} \leq e^{-\omega^6}$, for large enough $t$.

Thus, at this stage, we have proved that the core vertices, i.e., those in $[\k]$, all get infected \whp. 
If  no vertex outside the core has more than a single self-loop, then each vertex will have at least $m-1$ forward (i.e., out-going) edges. Hence, if $r\leq m-1$, the entire graph will be infected if the core is. We show that no vertex outside the core has more than one self-loop.

The probability that vertex $i$ outside the core has at least two self loops is at most $2\binom{m}{2}i^{-2}$. Summing over all $i \in [\k+1,t]$, this is  $O\brac{\int_{\omega^{1+\d/m}}^t i^{-2}\, \mathrm{d}i}= O\brac{1/\omega^{1+\d/m}}=o(1)$.  Hence, \whp, no vertex outside the core has more than one self-loop. So if $r\leq m-1$, the graph entire graph gets infected \whp.
\end{proof}

%========================================================================================================================
%========================================================================================================================
%========================================================================================================================

\section{Subcritical Case: Proof of Theorem~\ref{Subcritical case}}
The general proof strategy of Theorem~\ref{Subcritical case} is based on the following argument. 
Suppose that a vertex $i$ is not infected at round $\t=0$, but it is infected at round $\t=1$. Then there must be $r$ edges connected to 
$i$ that also connect to vertices infected in round $\t=0$. Assuming that these edges connect to different neighbours, we have a depth-$1$
tree. Similarly, if $i$ gets infected in round $\t=d$, then there must be some underlying \emph{witness structure} which caused this.
In particular, it may be that there is an $r$-ary tree of depth $d$ wherein in round $\t=0$ all the leaves are 
infected  and no internal vertices are. We call this a \emph{witness tree}. More generally, such a structure may contain cycles. 
We shall deal with witness trees first before addressing more general witness structures. We use a first moment argument to show that 
witness structures of a certain depth do not exist \whp. 
Before doing so, we need to develop the estimates that will allow us to bound the number of occurrences of a certain graph 
as a subgraph of $\PA_t$.

We revert to the model $\PA_t(1,\d)$, which, for notational convenience, we shall write as $\PA_t$. We have $\g=\frac{1}{2+\d}$. 
\\

We begin by defining a sequence of polynomials $(Q_n(x))_{n \geq 1}$ where $Q_1(x)=x$ and $Q_{n+1}(x)=Q_n(x)(x+n)$ for $n \geq 1$. 
\\

\begin{lemma}\label{QLem}
Let $X_t$ be a random variable measurable with respect to $\PA_t$. Then,
\begin{equation*}
\E[X_{t-1}Q_n(D_i(t)+\d)]=\E[X_{t-1}Q_n(D_i(t-1)+\d)]\frac{t+(n-1)\g}{t-\g}.
\end{equation*}
\end{lemma}

\begin{proof}
\begin{align*}
&\E[X_{t-1}Q_n(D_i(t)+\d) \mid \PA_{t-1}]\\
&\quad = X_{t-1}\left[ \brac{1-\frac{D_i(t)+\d}{(2+\d)t-1}} Q_n(D_i(t-1)+\d)+ \frac{D_i(t)+\d}{(2+\d)t-1}Q_n(D_i(t-1)+\d+1)\right]\\
&\quad = X_{t-1}\left[Q_n(D_i(t-1)+\d)\left[ 1-\frac{D_i(t)+\d}{(2+\d)t-1} + \frac{D_i(t)+\d+n}{(2+\d)t-1}\right]\right].
\end{align*}
Now we take expectations on both sides and the lemma follows. 
\end{proof}

\begin{lemma}\label{UBProbLemma}
Suppose  $i<j_1, j_2, \ldots, j_k$ are vertices in $\PA_t(m,\d)$. Then
\begin{equation*}
\Pr\brac{j_1 \rightarrow i \cap j_2 \rightarrow i \cap \ldots \cap j_k \rightarrow i } \leq M^k \frac{1}{i^\g j_1^{1-\g}}\frac{1}{i^\g j_2^{1-\g}}\ldots\frac{1}{i^\g j_k^{1-\g}}
\end{equation*}
where $M=M(m, \d)$ is a constant that depends only on $m$ and $\d$. 
\end{lemma}

\begin{proof}
\begin{eqnarray*}
\E[\Ind_{\{j_1\rightarrow i\}}\Ind_{\{j_2\rightarrow i\}}\ldots\Ind_{\{j_k\rightarrow i\}} \mid \PA_{j_k-1}]
&=&\Ind_{\{j_1\rightarrow i\}}\Ind_{\{j_2\rightarrow i\}}\ldots\Ind_{\{j_{k-1}\rightarrow i\}} \frac{D_i(j_k-1)+\d}{(2+\d)j_k-1}\\
&=&\frac{\g}{j_k-\g}\brac{\prod_{s=1}^{k-1}\Ind_{\{j_s\rightarrow i\}}}   (D_i(j_k-1)+\d).
\end{eqnarray*}

Therefore, applying Lemma \ref{QLem} repeatedly, 
\begin{eqnarray*}
\E[\Ind_{\{j_1\rightarrow i\}}\Ind_{\{j_2\rightarrow i\}}\ldots\Ind_{\{j_k\rightarrow i\}}]
&=&\frac{\g}{j_k-\g}\E\left[  \brac{\prod_{s=1}^{k-1}\Ind_{\{j_s\rightarrow i\}}}   Q_1(D_i(j_k-1)+\d)\right]\\
&=&\frac{\g}{j_k-\g}\E\left[  \brac{\prod_{s=1}^{k-1}\Ind_{\{j_s\rightarrow i\}}}   Q_1(D_i(j_k-2)+\d)\right]\frac{j_k-1}{j_k-1-\g}\\
&=&\vdots\\
&=&\frac{\g}{j_k-\g} \E\left[\brac{\prod_{s=1}^{k-1}\Ind_{\{j_s\rightarrow i\}}}Q_1 (D_i(j_{k-1})+\d)\right]
\prod_{s=j_{k-1}+1}^{j_k-1}\frac{s}{s-\g}.
\end{eqnarray*}

Now,
\begin{eqnarray*}
\E\left[\brac{\prod_{s=1}^{k-1}\Ind_{\{j_s\rightarrow i\}}}(D_i(j_{k-1})+\d) \mid \PA_{j_{k-1}-1}\right]
&=& \brac{\prod_{s=1}^{k-2}\Ind_{\{j_s\rightarrow i\}}}\frac{D_i(j_{k-1}-1)+\d}{(2+\d)j_{k-1}-1}(D_i(j_{k-1}-1)+\d+1)\\
&=& \frac{\g}{j_{k-1}-\g}\brac{\prod_{s=1}^{k-2}\Ind_{\{j_s\rightarrow i\}}}Q_2(D_i(j_{k-1}-1)+\d).
\end{eqnarray*}

Thus, by repeated application of Lemma \ref{QLem},
\begin{align*}
&\E[\Ind_{\{j_1\rightarrow i\}}\Ind_{\{j_2\rightarrow i\}}\ldots\Ind_{\{j_k\rightarrow i\}}]\\
&\quad =\frac{\g}{j_k-\g}\frac{\g}{j_{k-1}-\g} \E\left[\brac{\prod_{s=1}^{k-2}\Ind_{\{j_s\rightarrow i\}}}Q_2(D_i(j_{k-1}-1)+\d)\right]
\prod_{s=j_{k-1}+1}^{j_k-1}\frac{s}{s-\g}\\
&\quad =\frac{\g}{j_k-\g}\frac{\g}{j_{k-1}-\g}\E\left[\brac{\prod_{s=1}^{k-2}\Ind_{\{j_s\rightarrow i\}}}Q_2(D_i(j_{k-2})+\d)\right]
\prod_{s=j_{k-2}+1}^{j_{k-1}-1}\frac{s+\g}{s-\g}\prod_{s=j_{k-1}+1}^{j_k-1}\frac{s}{s-\g}.
\end{align*}

This pattern continues until we get
\begin{align*}
&\E[\Ind_{\{j_1\rightarrow i\}}\Ind_{\{j_2\rightarrow i\}}\ldots\Ind_{\{j_k\rightarrow i\}}]\\
&\quad =\frac{\g}{j_k-\g}\frac{\g}{j_{k-1}-\g}\ldots \frac{\g}{j_1-\g} \prod_{s=j_1+1}^{j_{2}-1}\frac{s+(k-2)\g}{s-\g} \ldots \prod_{s=j_{k-3}+1}^{j_{k-2}-1}\frac{s+2\g}{s-\g}\prod_{s=j_{k-2}+1}^{j_{k-1}-1}\frac{s+\g}{s-\g}\prod_{s=j_{k-1}+1}^{j_k-1}\frac{s}{s-\g}\\
&\qquad \times \E[Q_k(D_i(j_1-1)+\d)].
\end{align*}

Applying Lemma \ref{QLem} repeatedly,
\begin{eqnarray*}
\E[Q_k(D_i(j_1-1)+\d)]=\E[Q_k(D_i(i)+\d)]\prod_{s=i+1}^{j_1-1}\frac{s+(k-1)\g}{s-\g},
\end{eqnarray*}
and
\begin{eqnarray*}
\E[Q_k(D_i(i)+\d)]&=&\brac{1-\frac{1+\d}{(2+\d)i-1}}Q_k(1+\d)+\frac{1+\d}{(2+\d)i-1}Q_k(1+\d+1)\\
&=&Q_k(1+\d)\frac{i+(k-1)\g}{i-\g}. 
\end{eqnarray*}
Thus, 
\begin{align}
&\E[\Ind_{\{j_1\rightarrow i\}}\Ind_{\{j_2\rightarrow i\}}\ldots\Ind_{\{j_k\rightarrow i\}}]\nonumber \\
&\quad = \prod_{s=i}^{j_1-1}\frac{s+(k-1)\g}{s-\g}\prod_{s=j_1+1}^{j_{2}-1}\frac{s+(k-2)\g}{s-\g} \ldots \prod_{s=j_{k-3}+1}^{j_{k-2}-1}\frac{s+2\g}{s-\g}\prod_{s=j_{k-2}+1}^{j_{k-1}-1}
\frac{s+\g}{s-\g}\prod_{s=j_{k-1}+1}^{j_k-1}\frac{s}{s-\g}\nonumber \\
&\qquad \times\frac{\g}{j_k-\g}\frac{\g}{j_{k-1}-\g}\ldots \frac{\g}{j_1-\g} Q_k(1+\d).\label{ujr78jk85686}
\end{align}

Observe that 
\begin{equation*}
\prod_{s=j_{k-1}+1}^{j_k-1}\frac{s}{s-\g}=\frac{\G(j_k)}{\G(j_{k-1}+1)}\frac{\G( j_{k-1}+1-\g)}{\G(j_k-\g)}
=\frac{\G(j_k)}{\G(j_k-\g)}\frac{\G(j_{k-1}+1-\g)}{\G(j_{k-1}+1)},
\end{equation*}
and similarly with the other product terms. Thus, the product will give us
\begin{align*}
&\frac{\G(j_k)}{\G(j_k-\g)}\frac{\G(j_{k-1}+1-\g)}{\G(j_{k-1}+1)}\frac{\G(j_{k-1}+\g)}{\G(j_{k-1}-\g)}
\frac{\G(j_{k-2}+1-\g)}{\G(j_{k-2}+1+\g)}\frac{\G( j_{k-2}+2\g  )}{\G( j_{k-2}-\g  )}\frac{\G(  j_{k-3}+1-\g )}{\G( j_{k-3}+1+2\g  )}
\ldots\\
&\ldots \frac{\G( j_2+(k-2)\g  )}{\G( j_2-\g  )}\frac{\G( j_1+1-\g  )}{\G( j_1+1+(k-2)\g  )}\frac{\G( j_1+(k-1)\g  )}{\G( j_1-\g  )}
\frac{\G( i-\g  )}{\G( i+(k-1)\g)  }.
\end{align*}

Observe,
\begin{equation*}
\frac{\G(j_{k-1}+1-\g)}{\G(j_{k-1}+1)}\frac{\G(j_{k-1}+\g)}{\G(j_{k-1}-\g)}
=\frac{j_{k-1}-\g}{j_{k-1}}\frac{\G(j_{k-1}+\g) }{\G(j_{k-1})}< \frac{\G(j_{k-1}+\g) }{\G(j_{k-1})}.
\end{equation*}

A similar argument bounds the other fraction pairs, thereby giving an upper bound on the product of
\begin{equation}
\frac{\G(j_k)}{\G(j_k-\g)} \frac{\G(j_{k-1}+\g) }{\G(j_{k-1})} \frac{\G( j_{k-2}+2\g )}{\G(  j_{k-2}+\g )}\ldots
\frac{\G( j_1+(k-1)\g  )}{\G( j_1+(k-2)\g  )}\frac{\G( i-\g  )}{\G( i+(k-1)\g)  }.\label{g5gg7h9h}
\end{equation}

For some constant $c$ which depends only on $\g$, we have $\G(x+\g)/\G(x) \leq c x^{\g}$. Therefore, \eqref{g5gg7h9h} is bounded by $c^k\frac{(j_k, j_{k-1}\ldots j_1)^\g}{i^{k\g}}$. 

Going back to \eqref{ujr78jk85686}, we have
\begin{equation}
\E[\Ind_{\{j_1\rightarrow i\}}\Ind_{\{j_2\rightarrow i\}}\ldots\Ind_{\{j_k\rightarrow i\}}]
 \leq c^k \frac{1}{i^\g j_1^{1-\g}}\frac{1}{i^\g j_2^{1-\g}}\ldots \frac{1}{i^\g j_k^{1-\g}}
\end{equation} 
 where $c$ is a constant that depends only $\g$.  

We wish to extend the above result to $\PA_t(m,\d)$ with $m>1$, that is, we wish to bound the probability of the $j$s connecting
to the same vertex $i$. The $j$s do not have to be distinct. Recall that vertex $i$ in $\PA_t(m,\d)$  is created from grouping 
$m$ consecutive vertices in $\PA_{mt}(1,\d/m)$ and contracting them into one vertex (possibly creating loops and/or parallel edges
in doing so). 

Let $I=\{m(i-1)+1,m(i-1)+2, \ldots, mi \}$ be the set of vertices in $\PA_{mt}(1,\d/m)$ that group to become 
$i$ in $\PA_t(m,\d)$. Similarly, we have sets $J_1, J_2, \ldots, J_k$ for the $j$s. 

Then the event  $j_1 \rightarrow i$ in $\PA_t(m,\d)$ occurs in $\PA_{mt}(1,\d/m)$ when a vertex in $J_1$ throwing an edge to a vertex in $I$. This can happen in $m^2$ different ways. Then the event  
$j_1 \rightarrow i \cap j_2 \rightarrow i \cap \ldots\cap j_k \rightarrow i$ in $\PA_t(m,\d)$ can happen in at most $m^{2k}$ 
(it may be less than this since we do not insist the $j$s are distinct). Since edges to different vertices are negatively correlated (cf. Lemma~11.13 from~\cite{Remco} - Lemma~\ref{NegCorrLem} below), the probability is maximised when in $\PA_{mt}(1,\d/m)$  all the $j$s throw to the same vertex in $I$. 

\end{proof}

Recall that the notation $g(j,\ell)=i$ means the $\ell$th edge of vertex $j$ was thrown to vertex $i<j$. We use Lemma~11.13 
from~\cite{Remco} that states negative correlation between edges thrown to different vertices. For an integer $N_i$ let

\begin{equation*}
E_i=\bigcap_{n=1}^{N_i}\{g(j_n^{(i)}, \ell_n^{(i)}) = i\}.
\end{equation*}

Hence, $E_i$ denotes the event that a certain set of edges are thrown to vertex $i$. The following lemma says $E_i$ and $E_{i'}$ are negatively correlated for $i \neq i'$, that is, edges thrown to different vertices are negatively correlated. 
\\

\begin{lemma}[Lemma 11.13~\cite{Remco}]\label{NegCorrLem}
For distinct vertices $i_1, i_2, \ldots, i_k$ in $\PA_t(m,\d)$,
\begin{equation*}
\Pr\brac{\bigcap_{s=1}^k E_{i_s}} \leq \prod_{s=1}^k\Pr(E_{i_s}).
\end{equation*}
\end{lemma}

The following is a corollary of Lemmas  \ref{UBProbLemma} and \ref{NegCorrLem}. The $j$s need not be distinct, and some of the $j$s may also be $i$s. 
\\

\begin{corollary}\label{edgesJointProb}
Suppose $i_1, j_1, i_2, j_2, \ldots, i_k, j_k$ are vertices in $\PA_t(m,\d)$ where $i_s<j_s$ for $s=1,2,\ldots,k$. Then
\begin{equation*}
\Pr(j_1\rightarrow i_2 \cap j_2 \rightarrow i_2, \ldots, j_k \rightarrow i_k)
 \leq M^k\frac{1}{i_1^\g j_1^{1-\g}}\frac{1}{i_2^\g j_2^{1-\g}}\ldots \frac{1}{i_k^\g j_k^{1-\g}}
\end{equation*}
where $M=M(m,\d)$ is a constant that depends only on $m$ and $\d$. 
\end{corollary}

\subsection{Witness Trees}
In order to show that a vertex $i$ does not get infected in round $\t=1$ \whp, it suffices to show that that
there is no depth-$1$ witness structure, \whp. This can be done by showing that the expected number of such witness structures is
$o(1)$. We shall deal with trees first, where every internal (non-leaf) vertex has $r$ children.

Let $i \in \PA_t(m,\d)$. A particular tree $T_i$, rooted at $i=\text{root}(T_i)$ with leaves $L=\text{leaves}(T_i)$, is a subgraph of 
$\PA_t(m, \d)$. If  $L \subseteq \mathcal{I}_0$ but no other vertex is in $\mathcal{I}_0$, 
then $T_i$ is called a \emph{witness tree}. For the sake of the analysis, in this section 
it will be convenient to consider edges of $\PA_t(m, \d)$ to be directed, where edge $(i,j)$ is directed from the
younger to the older.  Thus, given a $T_i$ the orientations on its edges are already determined and we are not free
to alter them. Suppose vertex $j$ is a child of vertex $j'$ in a  tree $T_i$. If $j'<j$, then the edge $\{j',j\}$ is directed from $j$
to $j'$ and we call $(j,j')$ an \emph{up edge}; otherwise we call it an \emph{down edge}.

A given tree $T_i$ is a member of a rooted, directed, isomorphism class $\overrightarrow{\mathcal{T}_i}$: this consists of pairwise 
isomorphic rooted trees, where the root is labeled by $i$ and the other vertices have labels in 
$[t]\setminus \{i\}$. Here, we assume that an isomorphism between members of this class respects edge orientations. 

Alternatively, we may define $\overrightarrow{\mathcal{T}}$ to be a rooted, 
directed $r$-ary tree whose vertices are the variables
$x_0, x_1, x_2, \ldots, x_N$ and $x_0$ is the label/variable 
of the root. These variables take values in $[t]$. 
If we set $x_0=i$, then we denote the resulting tree (or class of trees) by $\overrightarrow{\mathcal{T}_i}$.
Every assignment of the variables which respects the edge orientations gives rise to a $T_i \in \overrightarrow{\mathcal{T}_i}$. 

Let $d_0=\min \{d \in \mathbb{N} : d\g>1\}$. As we shall see, we need only consider trees of depth at most $d_0$, which is, of course, a constant.
Consequently, there is a bounded number of isomorphism classes, and since each tree is $r$-ary, no tree has more than $r^{d_0+1}$
vertices. 

We shall deal with the cases $r\g>1$ and $r\g\leq 1$ separately, starting with the former. There, it suffices to consider only trees 
of depth 1. 

We require the following lemma.
\\

\begin{lemma}\label{noPar}
For infection probability $p=O(1/t^\g)$,  no vertex in $\mathcal{I}_0$ has parallel edges \whp.
\end{lemma}
\begin{proof}
Let $X_t^\parallel$ be a random variable that counts the number of vertices $j$ which throw parallel edges in $\text{PA}_t(m,\d)$. Then, dealing firstly with the case $\d<0$, 
\[
\E[X_t^\parallel] = O(1) \sum_{j=1}^t\sum_{i=1}^j\frac{1}{i^{2\g}j^{2(1-\g)}} = O(1) \sum_{j=1}^t\frac{1}{j^{2(1-\g)}}\int_1^j x^{-2\g} \, \mathrm{d}x= \frac{O(1)}{2\g-1}\sum_{j=1}^t\frac{1-j^{1-2\g}}{j^{2(1-\g)}}. 
\]
Now $1-2\g=\frac{\d/m}{2+\d/m}<0$ when $\d<0$.  Hence $\E[X_t^\parallel] =O(1) \int_1^t j^{-2(1-\g)}\, \mathrm{d}j=O(t^{2\g-1})$. 
Therefore, the expected number of vertices that are in $\mathcal{I}_0$ and throw parallel edges, or throw parallel edges to vertices in $ \mathcal{I}_0$, is $O(t^{\g-1})=o(1)$. 

When $\d=0$, we have $\g=1/2$ so the integral is $O((\log t)^2)$, giving probability $O((\log t)^2/t^\g)=o(1)$. 

When $\d>0$, we have $0<\g<1/2$ giving probability $O(\log t/t^\g)=o(1)$.
\end{proof}

\subsection{$r\g>1$}\label{subSecA}
In this section we prove Theorem~\ref{Subcritical case}\textbf{(i)}. We remind that $\d<0$ implies $r\g>1$. 

By Lemma \ref{noPar}, any vertex $i$ infected in round $\t=1$ must be infected by a depth-$1$ witness tree. 
\\

\begin{lemma}
Suppose $p=\frac{1}{\omega t^\g}$. For a vertex $i \in \PA_t(m,\d)$, the expected number of depth-$1$ witness trees rooted at 
$i$ is $O\bfrac{1}{\omega^ri^{r\g}}$. 
\end{lemma}
\begin{proof}
Let $T_i$ be such a tree with $k$ up edges and $r-k$ down edges. Specifically, say the up leaves are vertices $j_1, j_2, \ldots, j_k$ 
and the down leaves are vertices $j_{k+1}, j_{k+2}, \ldots, j_r$. If $T_i \subseteq \PA_{t}(m,\d)$ 
means that \emph{$T_i$ is a subgraph of} $\PA_t(m,\d)$, then
\begin{align*}
&\Pr\brac{T_i \subseteq \PA_{t}(m,\d)} \leq M^k \frac{1}{i^\g j_1^{1-\g}}\frac{1}{i^\g j_2^{1-\g}}\ldots\frac{1}{i^\g j_k^{1-\g}}\times \frac{1}{i^{1-\g}j_{k+1}^\g}\frac{1}{i^{1-\g}j_{k+1}^\g}
\ldots \frac{1}{i^{1-\g}j_{r}^\g}\\
&=M^k\frac{1}{i^{\g k+(1-\g)(r-k)}}\frac{1}{(j_1j_2 \ldots j_k)^{1-\g}}
\frac{1}{(j_{k+1}j_{k+1}\ldots j_r)^\g}.
\end{align*}

Therefore, the expected number of trees in the isomorphism class (i.e., those trees  isomorphic to $T_i$, rooted at $i$ and having 
the same edge orientations) is bounded from above by 
\begin{eqnarray*}
O(1)\frac{1}{i^{\g k+(1-\g)(r-k)}}\brac{\int_1^t j^{-1+\g}\,\mathrm{d}j}^k\brac{\int_1^i j^{-\g}\,\mathrm{d}j}^{r-k}
&=&O(1)\frac{1}{i^{\g k+(1-\g)(r-k)}}t^{\g k}i^{(1-\g)(r-k)}\\
&=& O(1)\bfrac{t}{i}^{\g k}.
\end{eqnarray*}
The above is therefore maximised when $k=r$, that is, when all edges to leaves are up. There are $2^r$ possible edge orientations, 
hence, multiplying by the probability that all leaves of such a tree are infected we get a bound of $O(1/\omega^r)i^{-r\g}$ for
the expected number of depth-$1$ witness trees rooted at $i$.
\end{proof}

The proof of Theorem \ref{Subcritical case}\textbf{(i)} is a corollary of the above: summing $O(1/\omega^r)i^{-r\g}$  over all $i$ from $1$ to $t$, the condition $r\g>1$ ensures we get $o(1)$.

\subsection{$r\g\leq 1$}
In this section we prove Theorem \ref{Subcritical case}\textbf{(ii)} and \textbf{(iii)}.

Recall $d_0=\min \{d \in \mathbb{N}: d\g>1\}$. We shall consider witness trees with depth at most $d_0$. Since $d_0$ is a constant
and each internal vertex has precisely $r$ children, there is only a bounded number of isomorphism classes. 

% To elaborate, a tree $T_i$ is a witness tree if it satisfies three criteria: (i) it occurs as a subgraph of $\PA_t(m,\d)$, (ii)
% all of its leaves are infected in round $\t=0$,  and (iii) none of its internal vertices are infected in round  $\t=0$. 

In round $\t=0$, there are $\Theta(t^{1-\g}/\omega)$ infected vertices in expectation. We will show that in expectation there are
$o(t^{1-\g}/\omega)$ newly infected vertices in each of rounds $\t=1,2,\ldots, d_0-1$, and $o(1)$ in round $d_0$. Consequently, 
the progression of the outbreak stops at or before round $d_0$ \whp\ and, moreover, by Markov's inequality, it follows that 
$|\mathcal{I}_f|/|\mathcal{I}_0| \stackrel{p}{\rightarrow} 1$ as $t\rightarrow \infty$. 

If $i$ gets infected in round $\t=d$, it must be the case that there is a depth-$d$ witness structure which causes this infection. 
We shall bound from above the expected number of such witness structures for $d=1,2,\ldots,d_0$. 
In this section, we focus on witness structures that are trees - the general case is treated in the next section. 

For the purposes of the next section, we will consider an \emph {extended isomorphism class} $\overrightarrow{\mathcal{T}}$ which 
is a rooted, oriented tree, whose vetrices are variables, taking values in $[t]$, and every vertex has \emph{at most} $r$ children. 
Assuming that the tree has $N+1$ vertices, the variables that are the labels of the vertices are $x_0, x_1, \ldots, x_N$, where 
$x_0$ is the label of the root vertex. 
When these variables are assigned values in $[t]$ that are compatible with the direction of the edges of $\overrightarrow{ \mathcal{T}}$
and the corresponding edges are present in $\PA_t(m,\d)$, then we have a realisation of $\overrightarrow{ \mathcal{T}}$. 
Thus, we may view $\overrightarrow{\mathcal{T}}$ as the set of all such realisations - we write $T \in \overrightarrow{\mathcal{T}}$. 
For any $i \in [t]$, we let $\overrightarrow{\mathcal{T}_i}$ denote the restriction of $\overrightarrow{\mathcal{T}}$ where the root 
variable $x_0$ has been set to $i$.

Let us consider, in particular, the case of a directed isomorphism class $\overrightarrow{\mathcal{T}_i}$, for some $i \in [t]$. 
Let $X_{\overrightarrow{\mathcal{T}_i}}$ count the
number of trees $T_i \in \overrightarrow{\mathcal{T}_i}$ such that $T_i \subseteq \PA_t(m,\d)$ and $L=\text{leaves}(T_i) \subseteq \mathcal{I}_0$. We have
\begin{equation}
\E\left [X_{\overrightarrow{\mathcal{T}_i}}\right] = \bfrac{1}{\omega t^\g}^{|L|}\sum_{T_i \in \overrightarrow{\mathcal{T}_i}}\Pr\brac{T_i \subseteq \PA_t(m,\d)}. \label{mainEq}
\end{equation}

Since each tree has at most $r^{d_0+1}$ edges,  by Corollary \ref{edgesJointProb} we have  
\begin{equation*}
\sum_{T_i \in \overrightarrow{\mathcal{T}_i}}\Pr\brac{T_i \subseteq \PA_t(m,\d)}
\leq C_1(m,\d,r) \sum_{T_i \in \overrightarrow{\mathcal{T}_i}}\prod_{(a,b) \in E(T_i)}\frac{1}{b^\g a^{1-\g}}  .
\end{equation*}
where $C_1(m,\d,r)$ is some constant that depends only on $m, \d, r$ and $E(T_i)$ is the edge set of $T_i$. 
(Recall that each edge is oriented from the largest vertex to the smallest one.)

Since each tree $T_i$ is an assignment of the variables $x_1, x_2, \ldots, x_N$
of the tree $\overrightarrow{\mathcal{T}_i}$ (recall that $x_0=i$), to calculate the sum \eqref{mainEq} we can perform a sum 
over all valid assignments. Our aim is to bound from above the above sum.

To this end, we consider a more general setting in which 
each vertex $x_a$ is associated with a \emph{valuation function} $v_a : [t] \rightarrow \mathbb{R}^+$.
When the variables/vertices $x_a$ assume some value, i.e., some assignment of a vertex in $[t]$, then the corresponding vertices get the value $v_a (x_a)$. 
We consider valuation functions of a certain form, namely, $v_a (j) = (\log j)^{\rho_a}/j^{e_a}$, where $\rho_a$ is a non-negative
integer and $e_a$ is a non-negative real number such that either $e_a = \rho_a = 0$ or, if $e_a >0$, then
$e_a = A \gamma + B(1-\gamma)$ with $A, B$ being non-negative integers that satisfy
$$ c(x_a)+A+B \geq r. \  \ \mbox{Property (A)},$$
where $c (x_a)$ denotes the number of children of $x_a$. 
In the former case, that is, when $e_a = 0$, we call the valuation function \emph{trivial}. Hence, if $x_a$ is 
an internal vertex with a trivial valuation function, then it has exactly $r$ children. 
If $x_a$ is a leaf, then either $v_a (j) = \frac{1}{\omega t^\gamma}=p$ (for $r=2$ we take $\omega =\log t$) or $v_a (j)= (\log j)^{\rho_a}/j^{e_a}$, where 
$\rho_a$ is a non-negative integer and $e_a > 0$ satisfies Property (A). In the former case, we call the leaf \emph{original}; otherwise,
we call it a \emph{contraction} leaf. 
The purpose of having a valuation function of this form will become apparent in the next section, where we consider general witness 
structures that are not trees. In those cases we perform a series of operations that convert a general witness structure into a tree. 
During these operations, we perform contractions of subtrees (hence the term \emph{contraction leaf}). 
Effectively, the valuation function is (up to multiplicative constants) the probability that the vertex/root of the contracted subtree
is infected through this subtree. 
When a leaf is original, it is meant to be externally infected, whereas 
a contraction leaf is infected through a certain sub-tree (that had been) rooted at it.  

For a vertex $x_a \in \overrightarrow{\mathcal{T}}$, we define the function $f_a: [t] \rightarrow \mathbb{R}_{\geq 0}$ recursively: 
If $x_a$ is a leaf, then $f_a(j)=v_a (j)$. 
Otherwise, with $x_{a_1}, x_{a_2}, \ldots, x_{a_k}$ being the 
child variables of $x_a$, where $k\leq r$, with $x_{a_1}, \ldots x_{a_{k_1}}$ having up edges with $x_a$ and the rest down, we set
\begin{equation}
f_a (j) := v_a (j) \brac{\prod_{s=1}^{k_1}\sum_{j'=j+1}^tf_{a_s}(j')\frac{1}{j^\g j'^{1-\g}}}
\brac{\prod_{s=k_1+1}^k\sum_{j'=1}^{j-1}f_{a_s}(j')\frac{1}{j'^\g j^{1-\g}}}.\label{giojho}
\end{equation}
We call $f_a$ the \emph{weight function} of the sub-tree that is rooted at $x_a$. 

However, if the valuation functions of the internal vertices are trivial and all leaves are original, then simply 
\begin{equation*}
f_0(i) \geq  \bfrac{1}{\omega t^\g}^{|L|} \sum_{T_i \in \overrightarrow{\mathcal{T}_i}}
\prod_{(a,b) \in E(T_i)}\frac{1}{b^\g a^{1-\g}}.
\end{equation*}
Thus by (\ref{mainEq}), if we show that 
\begin{equation} \label{eq:ToProve} 
\sum_{i=1}^t f_0 (i)  = o(1),
\end{equation}
this will imply that \whp\ there are no vertices which are infected through a tree that is isomorphic to 
$\overrightarrow{\mathcal{T}}$. 
 
To this end, we will first provide an upper bound on $f_0$ (cf. Lemma~\ref{fLemma} and Corollary~\ref{f0} below). 
In fact, we will provide a more general upper bound that is applicable to a general configuration of witness trees. 
This general form will be useful in the next section where we analyse the expected number of occurrences of 
general witness structures. 

% \textbf{Remark}
% Suppose $x_a$ is an internal vertex with some of its children leaves. Each leaf with an up edge contributes a factor 
% $O(1)\bfrac{t}{j}^\g$
% to the product \eqref{giojho}, and each leaf with a down edge contributes $O(1)$. Therefore, for an upper bound, 
% we may assume that all leaves have up edges. 

Let $\overrightarrow{\mathcal{T}} (a)$ denote the subtree of $\overrightarrow{\mathcal{T}}$ rooted at $x_a$. 
Thus, in particular, $\overrightarrow{\mathcal{T}} (0)= \overrightarrow{\mathcal{T}}$. 
Let $\rho(a)$ be the number of down edges in $\overrightarrow{\mathcal{T}} (a)$ and let
$\ell(a)$ be the number of its original leaves. We shall also be writing 
$a' \in \overrightarrow{\mathcal{T}}(a)$ to denote that $x_{a'}$ is a vertex of $\overrightarrow{\mathcal{T}}(a)$. 
For each $a' \in \overrightarrow{\mathcal{T}}$, we denote by $e_{a'}$ and $\rho_{a'}$ the exponents of the valuation function $v_{a'}$ 
of $x_{a'}$, as described above.  
\\ 

\begin{lemma}\label{fLemma}
Suppose that $\overrightarrow{\mathcal{T}}$ is as above. 
Let $x_a$ be an internal vertex of $\overrightarrow{\mathcal{T}}$.
%let $x_{a_1}, x_{a_2}, \ldots, x_{a_\ell}$ be the child variables of $x_a$, with $x_{a_1}, \ldots x_{a_k}$ having up edges with $x_a$ 
%and the rest down. 
Then uniformly for all $j \in [t]$, 
\begin{equation}
f_a(j) \lesssim \frac{1}{\omega^{\ell (a)}}~ \frac{(1 \vee (\log j)^{\rho (a)'})}{j^{y_a}},\label{CountUpperBound}
\end{equation}
where $y_a$ is such that either $y_a = A\gamma + B(1 - \gamma)$ where $A, B$ are non-negative integers that satisfy $A+B\geq r$ and $B>0$
or $y_a  \geq  \ell (a)\gamma + \sum_{a' \in \overrightarrow{\mathcal{T}} (a)} e_{a'}$,
and $\rho (a)' = \rho(a) + \sum_{a' \in \overrightarrow{\mathcal{T}} (a)} \rho_{a'}$. 
\end{lemma}
\textbf{Remark} The hidden constant factor in \eqref{CountUpperBound} depends only on $m$, $\d$ and $r$. 

\begin{proof}
In the following, for the sake of notational convenience we shall write $(\log j)^\rho$ instead of  $(1 \vee (\log j)^{\rho })$.  
We also set $\ell=\ell(a)$ and $\rho=\rho(a)$. 

We shall give a proof by induction starting from the bottom and going ``up'' the tree. Suppose that the children of an internal vertex
$x_a$ are all leaves with the edges that join them with $x_a$ pointing either upwards or downwards. 
% Suppose also that for those edges that point downwards the corresponding leaves are contraction leaves, which means that their 
% valuation function has the form $(\log j)^\rho /j^e$, where $e>0$ satisfies Property (A). 

Let us set $x_a=j$. 
Let $\mathcal{L}_1$ denote the subset of the indices of those leaves that are connected to $x_a$ through an up edge. Similarly, 
let $\mathcal{L}_2$ denote the subset of the indices of those leaves that are connected to $x_a$ through a down edge. 
We have 
\begin{equation} \label{eq:f_leaves} 
f_a (j) = v_a (j) \prod_{a' \in \mathcal{L}_1} \sum_{j_{a'} = j+1}^{t} v_{a'}( j_{a'}) \frac{1}{j^\gamma j_{a'}^{1-\gamma}} 
\prod_{a' \in \mathcal{L}_2} \sum_{j_{a'} = 1}^{j-1} v_{a'}(j_{a'} ) \frac{1}{j_{a'}^{\gamma}j^{1-\gamma}}.
\end{equation} 

The upper bound on each one of the above sums depends on the form of the valuation function as well as on the direction 
of the corresponding edge. 
The following claim provides this case distinction.
\begin{claim} \label{clm:sums_cases}
Assume that $x_{a'}$ is a leaf and let $f_{a'}(j') = (\log j')^{\rho_{a'}}/j'^{e_{a'}}$, if it is a contraction leaf, where $e_{a'}>0$ 
satisfies Property (A). Then for any $1\leq j <t$ we have 
\begin{equation*}
\sum_{j'=j+1}^t f_{a'}(j') {1 \over j^\gamma j'^{1-\gamma}} \lesssim 
\begin{cases} 
(\log j)^{\rho_{a'}} j^{-e_{a'}}, & \mbox{if $x_{a'}$ is a contraction leaf} \\
\frac{1}{\omega}~\frac{1}{j^\gamma}, & \mbox{if $x_{a'}$ is an original leaf}
\end{cases},
\end{equation*}
and 
\begin{equation*} 
\sum_{j'=1}^{j-1} f_{a'} (j') \frac{1}{j'^{\gamma} j^{1-\gamma}} \lesssim 
\begin{cases} 
 (\log j)^{\rho_{a'}} j^{-e_{a'}}, & \mbox{if $1-\gamma > e_{a'}$} \\
\frac{(\log j)^{\rho_{a'} + 1}}{j^{1-\gamma}}, & \mbox{otherwise} 
\end{cases}. 
\end{equation*}
%(Recall that in this case $x_{a'}$ is a contraction leaf.)
\end{claim}
\begin{proof}[Proof of Claim~\ref{clm:sums_cases}]  
The first sum is bounded from above by an integral: 
\begin{equation*} 
\sum_{j'=j+1}^t f_{a'}(j') {1 \over j^\gamma j'^{1-\gamma}} \leq  
\frac{1}{j^{\gamma}} \int_{j}^t v_{a'}(x) \frac{1}{x^{1-\gamma}} dx. 
\end{equation*}
Assume that $x_{a'}$ is a contraction leaf. In this case, the above integral becomes 
\begin{equation*}  
 \int_{j}^t v_{a'}(x) \frac{1}{x^{1-\gamma}} dx =  \int_{j}^t \frac{(\log x)^{\rho_{a'}}}{x^{1-\gamma + e_{a'}}} dx.
\end{equation*}
The value of this integral now depends on the sign of $-\gamma + e_{a'}$. Recall that $e_{a'}>0$ and it satisfies Property (A). 
Assume that $e_{a'}= A \gamma + B(1-\gamma)$. If $B \leq 1$, then $A \geq r - B$, whereby $\gamma - e_{a'} < 0$. If $B >  1$, then 
$\gamma - e_{a'} \leq \gamma -  2(1-\gamma) = 3\gamma - 2 < 0$, as $\gamma \leq 1/2$. 
Hence, by Lemma~\ref{lem:comb_integral} in Section~\ref{sec:integral}
\begin{equation*}  
 \int_{j}^t v_{a'}(x) \frac{1}{x^{1-\gamma}} dx \lesssim  (\log j)^{\rho_{a'}} j^{\gamma - e_{a'}},
\end{equation*}
and therefore 
\begin{equation*} 
\sum_{j'=j+1}^t f_{a'}(j') {1 \over j^\gamma j'^{1-\gamma}} \lesssim (\log j)^{\rho_{a'}} j^{-e_{a'}}. 
\end{equation*}
Assume now that $x_{a'}$ is an original leaf. In this case,
\begin{equation*}
\sum_{j'=j+1}^t f_{a'}(j') {1 \over j^\gamma j'^{1-\gamma}} \leq  
\frac{1}{j^{\gamma}} \int_{j}^t v_{a'}(x) \frac{1}{x^{1-\gamma}} dx = p \frac{1}{j^{\gamma}} 
\int_{j}^t  \frac{1}{x^{1-\gamma}} dx \lesssim p \frac{1}{j^{\gamma}} t^{\gamma} = \frac{1}{\omega}~\frac{1}{j^{\gamma}}. 
\end{equation*}

Consider now the second sum. If $x_{a'}$ is an original leaf, then 
\begin{equation*} 
\begin{split}
\sum_{j'=1}^{j-1} f_{a'}(j') {1 \over j'^{\gamma} j^{1-\gamma}} &\leq  
\frac{1}{j^{1-\gamma}} \int_{1}^j v_{a'}(x) \frac{1}{x^{\gamma}} dx  = 
p~\frac{1}{j^{1-\gamma}} \int_{1}^j \frac{1}{x^{\gamma}} dx \lesssim 
\frac{1}{\omega t^{\gamma}}~\frac{1}{j^{1-\gamma}}~j^{1-\gamma} \stackrel{j\leq t}\leq \frac{1}{\omega}~\frac{1}{j^\gamma}.
\end{split}
\end{equation*}
If $x_{a'}$ is a contraction leaf, then we have 
\begin{equation*} 
\begin{split}
\sum_{j'=1}^{j-1} f_{a'}(j') {1 \over j'^{\gamma} j^{1-\gamma}} &\leq  
\frac{1}{j^{1-\gamma}} \int_{1}^j v_{a'}(x) \frac{1}{x^{\gamma}} dx  = \frac{1}{j^{1-\gamma}} \int_{1}^j\frac{(\log x)^{\rho_{a'}}}{x^{\gamma +e_{a'}}} dx. 
\end{split}
\end{equation*}
If $1-\gamma > e_{a'}$, then the above becomes
\begin{equation*}
\sum_{j'=1}^{j-1} f_{a'}(j') {1 \over j'^{\gamma} j^{1-\gamma}}\lesssim \frac{1}{j^{1-\gamma}} (\log j)^{\rho_{a'}} j^{1-\gamma -e_{a'}} =(\log j)^{\rho_{a'}} j^{-e_{a'}}.
\end{equation*}
If $1-\gamma \leq e_{a'}$, then we will get
\begin{equation*}
\sum_{j'=1}^{j-1} f_{a'}(j') {1 \over j'^{\gamma} j^{1-\gamma}}\lesssim \frac{(\log j)^{\rho_{a'}+1}}{j^{1-\gamma}}.
\end{equation*}
 \end{proof}

\noindent
{\bf Remark} The above proof effectively shows that it is enough to consider only the case where original leaves are 
connected to their parent through an up edge. Thus, we may assume that original leaves are connected to their parent through up edges.

We will bound the two products of (\ref{eq:f_leaves}) using the above claim. 
Let us consider the first product. Let $\mathcal{L}_1'$ denote the subset of $\mathcal{L}_1$ that consists of the indices of 
the original leaves.  Let $\mathcal{L}_1''$ denote the complement of $\mathcal{L}_1'$ in $\mathcal{L}_1$ - this consists of 
the indices of the contraction leaves. 
Hence, applying the first part of the above claim we obtain: 
\begin{equation*} 
\begin{split}
\prod_{a' \in \mathcal{L}_1} \sum_{j_{a'} = j+1}^{t} v_{a'}( j_{a'}) \frac{1}{j^\gamma j_{a'}^{1-\gamma}}  
&= \prod_{a' \in \mathcal{L}_1'} \sum_{j_{a'} = j+1}^{t} v_{a'}( j_{a'}) \frac{1}{j^\gamma j_{a'}^{1-\gamma}} 
\prod_{a' \in \mathcal{L}_1''} \sum_{j_{a'} = j+1}^{t} v_{a'}( j_{a'}) \frac{1}{j^\gamma j_{a'}^{1-\gamma}} \\
&\lesssim \left( \frac{1}{\omega} \right)^{|\mathcal{L}_1'|}~\frac{ (\log j)^{\sum_{a' \in \mathcal{L}_1''}\rho_{a'}}}
{j^{|\mathcal{L}_1'|\gamma + \sum_{a' \in \mathcal{L}_1''}  e_{a'}}}.
\end{split}
\end{equation*}
Similarly, we deduce an upper bound on the second product through the second part of the claim. Here, we split 
$\mathcal{L}_2$ into two sets: let $\mathcal{L}_2'$ be the set of indices of those leaves for which $1-\gamma > e_{a'}$ and 
$\mathcal{L}_2''$ the complement of this set in $\mathcal{L}_2$. Hence, we have 
\begin{equation*}
\begin{split} 
\prod_{a' \in \mathcal{L}_2} \sum_{j_{a'} = 1}^{j-1} v_{a'}(j_{a'} ) \frac{1}{j_{a'}^{\gamma}j^{1-\gamma}} \lesssim 
\frac{(\log j )^{|\mathcal{L}_2''| + \sum_{a' \in \mathcal{L}_2} \rho_{a'}}}
{j^{\sum_{a' \in \mathcal{L}_2'} e_{a'} + (1-\gamma )|\mathcal{L}_2''|}}. 
\end{split}
\end{equation*}
Thus, (\ref{eq:f_leaves}) now yields: 
 \begin{equation*} \label{eq:f_leaves_bound}
 f_a (j) \lesssim  
\left( \frac{1}{\omega} \right)^{|\mathcal{L}_1'|}~\left( \frac{ 
\left( \log j \right)^{|\mathcal{L}_2''| + \sum_{a' \in \mathcal{L}_2 \cup \mathcal{L}_1''} \rho_{a'}}}
{j^{|\mathcal{L}_1'|\gamma + \sum_{a' \in \mathcal{L}_1'' \cup \mathcal{L}_2'}e_{a'}  + (1- \gamma)|\mathcal{L}_2''| + e_a} 
}\right).  
 \end{equation*}
Let us consider the exponent of $j$ which we denote by $y_a$. 
If $\mathcal{L}_1'' \cup \mathcal{L}_2' = \emptyset$, then the exponent is equal to 
$|\mathcal{L}_1'|\gamma + (1- \gamma)|\mathcal{L}_2''| + e_a$. But $|\mathcal{L}_1'|+ |\mathcal{L}_2''| = c (x_a)$ and since 
$e_a$ satisfies Property (A), it follows that $|\mathcal{L}_1'|\gamma + (1- \gamma)|\mathcal{L}_2''| + e_a = A\gamma + B(1-\gamma)$, 
where $A, B$ are non-negative integers that satisfy $A+B \geq r$. 

If $\mathcal{L}_1'' \cup \mathcal{L}_2' \not = \emptyset$, then $e_{a'} > 0$ for some $a' \in \mathcal{L}_1'' \cup \mathcal{L}_2'$. 
But this satisfies Property (A) and since $c (x_{a'})=0$, it follows that $e_{a'}= A\gamma + B(1-\gamma)$ for some non-negative integers 
$A, B$ that satisfy $A+B \geq r$. Thereby, the whole sum satisfies this. 

Assume now, that $|\mathcal{L}_1'|\gamma + \sum_{a' \in \mathcal{L}_1'' \cup \mathcal{L}_2'}e_{a'}  + (1- \gamma)|\mathcal{L}_2''| + e_a 
$ cannot be expressed in the form $A\gamma + B(1-\gamma)$ with $B>0$. Then necessarily $|\mathcal{L}_2''|=0$ and $\mathcal{L}_2'= \mathcal{L}_2$.  
Also, it is clear that $|\mathcal{L}_1'| =\ell (a)$. 
Hence, it follows that 
$$ y_a = \ell (a) \gamma + \sum_{a' \in \mathcal{L}_1'' \cup \mathcal{L}_2}e_{a'} + e_a = \ell (a) \gamma + \sum_{a' \in \overrightarrow{\mathcal{T}_i} (a)} e_{a'}. $$
This concludes the base case of the induction. 

Now we consider the case where some of the children of $x_a$ are not leaves. In general, some of these children 
are connected to $x_a$ by up edges and the rest by down edges. 
We consider each case separately. Let us assume that $x_a = j$. 

Assume that $x_{a_1}$ is a child of $x_a$ that is an internal vertex.
Letting $\ell_1=\ell(a_1)$, the number of original leaves in the subtree rooted at $x_{a_1}$, we have by the induction hypothesis,
$f_{a_1}(j_1)\lesssim \bfrac{1}{\omega}^{\ell_1}~\frac{(\log j_1)^{\rho_1}}{j_1^{A_1\g+B_1(1-\g)}}$ for some appropriate 
$A_1, B_1$ and $\rho_1$ as in the statement of the lemma. 
In particular, these are such that $A_1+B_1 \geq r$, provided that $A_1 + B_1  > 0$. 
% and moreover, if $B_1=0$, then $A_1 = \ell (a_1) + \sum_{a' \in } e_{a'}$.

Suppose that the child $x_{a_1}=j_1$ is connected by an up edge. We have
\begin{eqnarray*}
\sum_{j_1=j+1}^tf_{a_1}(j_1)\frac{1}{j^\g j_1^{1-\g}}  
&\lesssim &\bfrac{1}{\omega}^{\ell_1} \sum_{j_1=j+1}^t\frac{1}{j^\g j_1^{1-\g}}\frac{(\log j_1)^{\rho_1}}{j_1^{A_1\g+B_1(1-\g)}} \nonumber \\ %\label{qdg6hjuj} \\
&\lesssim &\bfrac{1}{\omega}^{\ell_1}~\frac{1}{j^\g}\int_j^t x_1^{-1+\g -(A_1\g+B_1(1-\g))}(\log x_1)^{\rho_1}\, \mathrm{d}x_1. 
\end{eqnarray*}
The last integral is bounded from above using Lemma~\ref{lem:comb_integral} from Section~\ref{sec:integral} giving 
\begin{eqnarray} \label{6f4s68d}
\sum_{j_1=j+1}^t f_{a_1}(j_1)\frac{1}{j^\g j_1^{1-\g}}  
\lesssim \bfrac{1}{\omega}^{\ell_1}~\frac{(\log j)^{\rho_1}}{j^{A_1\g+B_1(1-\g)}}.
\end{eqnarray} 
Observe $\g -(A_1\g+B_1(1-\g))<0$ in all possible cases: if $B_1=0$ then $A_1 \geq r \geq 2$; if $B_1=1$ then $A_1 \geq r-1 \geq 1$; and if $B_1 \geq 2$   then $\g -(A_1\g+B_1(1-\g))<0$ since $\d \geq 0 \Rightarrow$ $1-\g \geq \g$. 

Note that due to the fact that we consider trees of bounded degree and 
depth, terms such as $\rho$ and $A\g+B(1-\g)$ will always be bounded from above and below by constants that depend only on
$m$, $\d$ and $r$. Therefore, the constant factor incurred by the above integration is always bounded by
some constant that only depends on these parameters. 

Observe that \eqref{6f4s68d} is the same (up to multiplicative constants) as the expression for $f_{a_1}(j_1)$ except that $j$ 
has replaced $j_1$. In this sense, we
see that an up edge causes the parent vertex to ``reverse inherit'' the exponent of the child, in this case, that exponent being
$A_1\g +B_1(1-\g)$. 

Now we will consider what happens if it is a down edge, where, by assumption, $x_{a_1}$ is an internal vertex. We have
\begin{eqnarray} \label{eq:down_edge}
\sum_{j_1=1}^{j-1}f_{a_1}(j_1)\frac{1}{j^\g j_1^{1-\g}}
&\lesssim &\bfrac{1}{\omega}^{\ell_1}~\sum_{j_1=1}^{j-1}\frac{1}{j_1^\g j^{1-\g}}\frac{(\log j_1)^{\rho_1}}{j_1^{A_1\g+B_1(1-\g)}} \nonumber \\
 &\lesssim&\bfrac{1}{\omega}^{\ell_1}~\frac{(\log j)^{\rho_1}}{j^{1-\g}}\int_1^j x_1^{-\g -(A_1\g+B_1(1-\g))}\, \mathrm{d}x_1
 \nonumber \\
 &\lesssim& \left\{ 
  \begin{array}{l l}
   \bfrac{1}{\omega}^{\ell_1} \frac{(\log j)^{\rho_1+1}}{j^{1-\g}} & \quad \text{if  $1-\g-(\g A_1+ (1-\g)B_1)\leq 0$} \label{gh57jf} \\
    \bfrac{1}{\omega}^{\ell_1} \frac{(\log j)^{\rho_1}}{j^{A_1\g+B_1(1-\g)}} & \quad \text{if  $1-\g-(\g A_1+ (1-\g)B_1)>0$}%\label{gd468h4}
  \end{array} \right.
\end{eqnarray}
We observe that if $B_1 \geq 2$ then $1-\g-(\g A_1+ (1-\g)B_1)<0$; if $B_1=1$, then $A_1 \geq r-1 \geq 1$ so $1-\g-(\g A_1+ (1-\g)B_1)<0$; and if $B_1=0$, then $A_1 \geq \ell_1$ by the induction hypothesis. 

Once again, we emphasise that the integration incurs a constant factor that is bounded by a constant that depends only on 
$m$, $r$ and $\d$. 

Let $\mathcal{C}(a)$ denote the set of indices of the children of $x_a$. 
For any $a'\in \mathcal{C}(a)$, let $r_{a'}$ denote the exponent of $j$ in $f_{a'}(j)$.
We let $\mathcal{C}_1 \subseteq \mathcal{C}(a)$ denote the 
set of indices of the original leaves among the members of $\mathcal{C}(a)$. 
%We denote by $\mathcal{C}_1'$ the set of indices of the 
%contraction leaves in $\mathcal{C}(a)$
Also, we let $\mathcal{C}_2 \subseteq \mathcal{C}(a)$ denote the set 
of the indices of those children of $x_a$ that are not original leaves but are connected to $x_a$ through up edges. 
Let $\mathcal{C}_2'$ denote the set of the indices of those children that are not original leaves, are connected to $x_a$ through down 
edges and $1-\gamma > r_{a'}$, for $a' \in \mathcal{C}_2'$. Similarly, we define as $\mathcal{C}_2''$ the set of the indices of those
children that are not original leaves, are connected to $x_a$ through down edges but $1-\gamma \leq r_{a'}$, for $a' \in \mathcal{C}_2''$.

By Claim~\ref{clm:sums_cases} together with \eqref{6f4s68d} and \eqref{eq:down_edge}, we conclude that 
\begin{equation} \label{eq:f_a_bound} 
f_a(j) \lesssim v_a (j) \bfrac{1}{\omega}^{\sum_{a' \in \mathcal{C}(a)} \ell (a')} \frac{ (\log j )^{|\mathcal{C}_2''|
+\sum_{a' \in \mathcal{C}_2 \cup \mathcal{C}_2' \cup \mathcal{C}_2''} \rho (a')}}{j^{|\mathcal{C}_1|\gamma + \sum_{a' \in \mathcal{C}_2 \cup \mathcal{C}_2'} r_{a'} + (1-\gamma) |\mathcal{C}_2''|}}
\end{equation}
% Firstly, note that $|\mathcal{C}_2''| +\sum_{a' \in \mathcal{C}_2 \cup \mathcal{C}_2' \cup \mathcal{C}_2''} \rho (a')$
% is equal to the number of down edges in the sub-tree that is rooted at $x_a$. 
Let $y_a$ denote the exponent of $j$. Firstly, note that  $\ell (a) = \sum_{a' \in \mathcal{C}(a)} \ell (a')$. 

Assume that $\mathcal{C}_2 \cup \mathcal{C}_2' = \emptyset$. Then 
$y_a=|\mathcal{C}_1|\gamma  + (1-\gamma) |\mathcal{C}_2''| + e_a$. But as $|\mathcal{C}_1| + |\mathcal{C}_2''| = d(x_a)$ and 
$e_a$ satisfies Property (A), it follows that $|\mathcal{C}_1|\gamma  + (1-\gamma) |\mathcal{C}_2''| + e_a = A \gamma + B(1-\gamma)$, 
where $A, B$ are non-negative integers that satisfy $A+B \geq r$. 

If $\mathcal{C}_2 \cup \mathcal{C}_2' \not = \emptyset$, then $r_{a'} >0$, for some $a' \in \mathcal{C}_2 \cup \mathcal{C}_2'$, 
which has the form $A \gamma + B(1-\gamma)$, for some $A, B$ that are non-negative integers satisfying  $A+B \geq r$. 
Hence, the exponent of $j$ satisfies this as well. 

% We, finally, verify that the exponent of $j$ satisfies Property (A). 
% If $\sum_{a' \in \mathcal{C}_2 \cup \mathcal{C}_2'} e_{a'} + e_a>0$, 
% then by Claim~\ref{clm:additive} we deduce that the exponent of $j$
% satisfies Property (A). Assume now that this quantity is equal to 0. In this case, $|\mathcal{C}_1|+|\mathcal{C}_2''|=r$ and therefore
% Property (A) is also satisfied. 
%% Now, if $v_a$ is trivial, then there is nothing more to prove. Otherwise, $v_a$ is of the form $j^{-e_a}$, where $e_a$ satisfies Property (A), then, as $e_a$ is added to the exponent of $j$, Claim~\ref{clm:additive} implies that the final exponent satisfies Property (A). 

Assume now that $y_a$ cannot be written in the form $A\g + B(1-\g)$ with $A, B$ non-negative integers and $B>0$. 
Then this is the case for $r_{a'}$ for any $a' \in \mathcal{C}_2 \cup \mathcal{C}_2'$. Hence, by the induction hypothesis 
$\sum_{a' \in \mathcal{C}_2 \cup \mathcal{C}_2'} r_{a'}$ is equal to the number of original leaves that are
contained in the sub-tree that is rooted at those $x_{a'}$ together with the sum of the exponents $e_{a'}$ of the valuation functions 
of the vertices of these sub-trees. 
Moreover, $|\mathcal{C}_2''|=0$ and recall that $|\mathcal{C}_1|$ is the number of original leaves
that are directly connected to $x_a$. 
Thereby, 
$$y_a= \ell (a)\g + \sum_{a' \in \overrightarrow{\mathcal{T}} (a)} e_{a'}.$$  
\end{proof}

\vspace*{5mm}
The above lemma now implies the following. 
\begin{corollary} \label{f0}
If the valuation functions of the internal vertices of $\overrightarrow{\mathcal{T}}$ are trivial and all leaves are original, then 
\begin{equation}
f_0(i)\lesssim \bfrac{1}{\omega}^\ell \bfrac{(1 \vee (\log i)^\rho)}{i^{y_0}} \label{f_0(i)UB}
\end{equation}
where $\ell=\ell(0)$ and $\rho=\rho(0)$ and either $y_0=A\g + B(1-\g)$ where $A, B$ are non-negative integers that satisfy $A+B\geq r$ and 
$B>0$ or $y_0 =\ell \gamma$. 
\end{corollary}

\vspace*{5mm}

We conclude with the proof of \eqref{eq:ToProve} for depth $d_0$. 
Consider the expression on the right-hand side of~\eqref{f_0(i)UB}. If $0< B < r$ and $r \geq 3$, then $A\geq r-B$ and 
it is easy to check that $A\g + B(1-\g)>1$ (it is a convex combination of two positive numbers that are at least 1, one of which is bigger
than one, where $\gamma \not = 0,1$).
If $B \geq r$, then $A \g + B(1-\g) \geq r(1-\g)\geq 3 (1-\g)$. But $\g \leq 1/2$, whereby $3(1-\g) \geq 3/2 > 1$. 
If $y_0 = \ell \g$, then $\ell \geq d_0$ implies 
$y_0 \geq \g d_0>1$.

If $r=2$ then we are not necessarily guaranteed $A\g + B(1-\g)>1$ since, for example, $B=1$ and $r=2$  only assures
$\g A +(1-\g) B \geq 1$. 

If $\g A+(1-\g)B > 1$, then 
\begin{equation}\label{eq:exp_upper_bound}
 \sum_{i=1}^t f_0(i)\lesssim  \bfrac{1}{\omega}^\ell \sum_{i=1}^t \frac{ (1 \vee (\log i)^{\rho})}{i^{\g A + (1-\g)B}} 
\lesssim \bfrac{1}{\omega}^\ell \int_{1}^t\frac{(\log x)^{\rho}}{x^{\g A + (1-\g)B}} \, \mathrm{d}x
=O\left(\bfrac{1}{\omega}^\ell \right).
\end{equation}
In this case, the expected number of witness trees of this isomorphism class, over all $i$, is $o(1)$. 

In the case that $y_0=\ell \g \leq 1$, the sum is $\bfrac{1}{\omega}^\ell (\log t)^\rho t^{1-\ell \g}$ and the
expected number of witness trees of this isomorphism class, over all $i$, is $O((\log t)^\rho t^{1-\ell \g}/\omega^\ell)=o(t^{1-\g}/\omega^{\ell})$ since $\ell \geq 2$. 
In other words, the expected number of witness trees of depth less than $d_0$ is $o(t^{1-\g}/\omega^{\ell})$. 

As stated above, there are only a bounded number of isomorphism classes that we need to consider, hence the relevant constant factors 
are absorbed into the $O(.)$ terms above. 

We would like to extend this to include $r=2$, wherein if we take the depth of the tree to be equal to $d_0$, then it may be the case that
the exponent of $i$ in $f_0(i)$ is $1$ (which is the 
minimum it can be when $\ell\geq d_0$). In that case, the integral in (\ref{eq:exp_upper_bound}) would grow like $(\log t)^{\rho +1}$. 
To bypass this difficulty, when $r=2$ we consider witness trees that have depth equal to $d_0+1$. 
Recall that in this case we assume that $p_0 = \frac{1}{\log t}~\frac{1}{t^{\g}}$. Also, as we have already commented in the 
proof of Claim~\ref{clm:sums_cases}, we may assume that the witness trees we consider are such that all their leaves are 
connected to the rest of the tree through up edges. 

Let $x_1, x_2$ be the children of $x_0$ and assume without loss of generality that the subtree that is rooted at $x_1$ has depth 
$d_0$. 
Suppose the exponent of $j_1$ is $1$. That is, recalling that this subtree has $\ell (1)$ original leaves, by Lemma~\ref{fLemma}
we have $f_1 (j_1) \lesssim \frac{1}{(\log t)^{\ell (1)}} \frac{(\log j_1)^{\rho (1)}}{j_1}$. 
Thus, if $x_1$ is connected by an up edge with $x_0$, by (\ref{6f4s68d}) the exponent transfers, and we get a factor 
$\frac{1}{(\log t)^{\ell (1)}}\frac{(\log i)^{\rho (1)}}{i}$ in $f_0(i)$. 
If it is connected through a down edge, by (\ref{gh57jf}), we get the factor
\begin{equation*}
\frac{1}{(\log t)^{\ell (1)}}~\frac{1}{i^{1-\g}}\int_1^i x_1^{-\g-1}(\log x_1)^{\rho (1)} \, \mathrm{d}x_1 \lesssim 
\frac{1}{(\log t)^{\ell (1)}}~\frac{1}{i^{1-\g}}. 
\end{equation*}
If $x_2$ is an original leaf, then by Claim~\ref{clm:sums_cases} it contributes a factor that is at most (up to a 
multiplicative constant) $\frac{1}
{(\log t)}\frac{1}{i^\g}$,
thus giving in total
$\frac{1}{(\log t)^{\ell (1)+1}}\frac{1}{i^{1+\g}}$ or 
$\frac{1}{(\log t)^{\ell (1)+1}}\frac{1}{i}$. 
In any case,

\begin{equation}
f_0(i) \lesssim \frac{1}{(\log t)^{\ell(1)+1}}~\frac{1}{i}. \label{fupperbound1}
\end{equation}

If $x_2$ is not an original leaf, then by Lemma~\ref{fLemma}
$f_2 (j_2) \lesssim \frac{1}{(\log t)^{\ell (2)}}\frac{(\log j_2)^{\rho (2)}}{j_2^{y_2}}$, where either $y_2$ can be written as 
 $A \g + B(1-\g)$ for some non-negative integers $A, B$ that satisfy $B \geq 1$ and $A+B \geq 2$, or $y_2\geq \ell (2) \g$, where
 in this case $\ell(2) \geq 2$. 

If $x_2$ is joined to $x_0$ by an up edge, then by (\ref{6f4s68d}) it contributes a factor that is at most (up to a constant)
$\frac{1}{(\log t)^{\ell (2)}}~\frac{(\log i)^{\rho (2)}}{i^{y_2}}\lesssim \frac{1}{(\log t)^{\ell (2)}}~\frac{(\log i)^{\rho (2)}}{i^{2\g}}$, giving a 
total 
\begin{equation}
f_0(i) \lesssim   \frac{1}{(\log t)^{\ell (1)}}\frac{(\log i)^{\rho (1)}}{i^{1-\g}}   \frac{1}{(\log t)^{\ell (2)}}     \frac{(\log i)^{\rho (2)}}{i^{2\g}}=\frac{1}{(\log t)^{\ell(0)}}\frac{(\log i)^{\rho(0)}}{i^{1+\g}}. \label{fupperbound2}
\end{equation}

If $x_2$ is not an original leaf and is connected to $x_0$ by a down edge, the possibilities are 
\[
f_0(i) \lesssim \frac{1}{(\log t)^{\ell (1)}}\frac{(\log i)^{\rho (1)}}{i}\frac{1}{(\log t)^{\ell (2)}}\frac{(\log i)^{\rho (2)+1}}{i}, 
\]
or
\[
f_0(i) \lesssim \frac{1}{(\log t)^{\ell (1)}}\frac{(\log i)^{\rho (1)}}{i}\frac{1}{(\log t)^{\ell (2)}}\frac{(\log i)^{\rho (2)+1}}{i^{2\g}}, 
\]
or
\[
f_0(i) \lesssim \frac{1}{(\log t)^{\ell (1)}}~\frac{1}{i^{1-\g}}\frac{1}{(\log t)^{\ell (2)}}\frac{(\log i)^{\rho (2)+1}}{i}, 
\]
or
\[
f_0(i) \lesssim \frac{1}{(\log t)^{\ell (1)}}~\frac{1}{i^{1-\g}}\frac{1}{(\log t)^{\ell (2)}}\frac{(\log i)^{\rho (2)+1}}{i^{2\g}}.
\]

In all cases, 
\begin{equation}
f_0(i) \lesssim  \frac{1}{(\log t)^{\ell(0)}}\frac{(\log i)^{\rho(0)}}{i^{1+\g}}. \label{fupperbound3}
\end{equation}

Summing \eqref{fupperbound1}, \eqref{fupperbound2} or \eqref{fupperbound3} over $i=1,\ldots, t$ gives $o(1)$.

Consequently, the expected number of witness trees of depth $d_0+1$ when the initial infection
probability is $p =\frac{1}{(\log t) t^\g }$ is $o(1)$ as well. 

We have shown that if $r\geq 3$, then $\whp$ the process stops in \emph{less than} $d_0$ rounds, whereas for $r=2$ 
(with the appropriate choice of $p_0$)  it stops in \emph{at most} $d_0$ rounds. Note that $d_0 = \lfloor \frac{1}{\g} \rfloor + 1$. 
To be more precise, we have shown the bounds of Theorem~\ref{Subcritical case} for witness structures that are trees. We need to argue
about general witness structures that may contain cycles. In this case, we show that the expected number of occurrences of such a structure
is bounded by the expected number of occurrences of a tree that is appropriately constructed and has depth either $d_0$ or $d_0+1$, 
depending on the value of $r$.

\subsubsection{General witness structures}\label{General witness structures}
We consider witness structures that may have cycles. Recall that we are only considering the case $\d\geq 0$, 
since $\d<0 \Rightarrow r\g>1$.  

Firstly, the following lemma allows us to consider witness structures where the initially infected vertices are vertices which do not belong 
to cycles. 
\\
\begin{lemma}\label{NoInfectedOnSmallCycle}
Let $K$ be positive constant. If $p=O(1/t^\g)$, then with high probability, no initially infected vertex lies a cycle of 
of size at most $K$ 
\end{lemma}
\begin{proof}
For a cycle $C=(a_1,a_2,\ldots,a_k)$ of size $k \leq K$, we apply Corollary \ref{edgesJointProb}, 
\begin{equation*}
\Pr(C \subseteq \PA_t(m,\d)) \leq M^{2k}\prod_{i=1}^{k} \frac{1}{(a_i \wedge a_{i+1})^{\g}(a_i \vee a_{i+1})^{1-\g}}
\leq \frac{M^{2k}}{a_1\ldots a_k},
\end{equation*}
where we have used the fact that for $i<j$, $\frac{1}{i^\g j^{1-\g}} \leq \frac{1}{(ij)^{1/2}}$ when $\g \leq \frac{1}{2}$. 

Thus, the expected number of cycles in $\PA_t(m,\d)$ of size at most $K$ is bounded from above by
\begin{equation}
\sum_{3 \leq k \leq K}\sum_{a_1,\ldots, a_k}\frac{M^{2k}}{a_1\ldots a_k}=O((\log t)^K )
\end{equation}
and so the number of initially infected vertices on such cycles is $O((\log t)^{K+1} /t^\g)=o(1)$. 
\end{proof}

Recall that if a vertex $i$ becomes infected in round $\t$, then it must have been infected by some neighbours, at least one of which 
got infected in round $\t-1$. Iterating this argument, there must be a chain of infections of length $\t$ that started in a set of initially
infected vertices. This is witnessed by a rooted subgraph, whose root is vertex $i$ and whose other vertices can be classified according 
to their \emph{depths}. Let us consider this notion more precisely. 
Suppose $x$ and $y$ are neighbours in $\PA_t(m,\d)$, 
$x \in \mathcal{I}(\t) \cap \mathcal{S}(\t-1)$, and  $y \in \mathcal{I}(\t-1)$. 
%Then we say $y$ \emph{infected} $x$ (of course, there will, in general
%by more than one vertex that infected $x$). 
%If  $y$ infected $x$, then 
Then we say $x$ is a \emph{parent} of $y$ and $y$ a \emph{child} of $x$.
If $x$ is a parent of $y$ and $x<y$ then $\{ x,y\}$ is an up edge. If $x>y$, then it is a down edge. 
The notion of \emph{parent-child} gives rise to the \emph{depth} of a vertex. 
Let $\depth (i) = 0$ and $\depth(y) = 1 + \max \{ \depth (x) \ : \ \mbox{$y$ is a child $x$}  \}$. 
We shall use this notion later in our proof. 

Suppose a vertex $i \in \mathcal{I}(\t)$ for some $\t >0$. Then there must exist a subgraph $S_i \subseteq \PA_t(m,\d)$ such
that the following hold:

\begin{description}
\item[(1)] every vertex in $S_i$ except $i$ has a parent in $S_i$;
\item[(2)] the set $L = S_i \cap \mathcal{I}(0)$, which we call the set of \emph{leaves}, is non-empty;
\item[(3)] every parent in $S_i$ has exactly $r$ edges in $S_i$ which go to children in $S_i$.
\end{description}

If, furthermore, $i \in \mathcal{I}(\t) \cap \mathcal{S}(\t-1)$, that is, $i$ got infected in round $\t$, then we also have:
\begin{description}
\item[(4)]  $\depth (S_i) = \max_{j \in S_i}\depth(j)=\t$, where $\depth (i) =0$.
\end{description}

We call such an $S_i$ a \emph{witness structure rooted at $i$}. Observe that \textbf{(1)} forces $S_i$ to be connected. Of course, leaves
cannot be parents. Additionally, recall that we only need to analyse bounded size structures and, therefore, only have bounded size cycles.
Hence, by Lemmas \ref{noPar} and  \ref{NoInfectedOnSmallCycle}, any leaf will, with high probability,
 have degree $1$ in $S_i$. We will assume this to be the case. 
 
 Condition \textbf{(3)} implies that a parent has at most $r$ children in $S_i$, and the witness structure is a witness tree as
 per the previous definition, if and only if 
 every parent has exactly $r$ children and every vertex except $i$ has exactly one parent. For a tree, it is also 
 the case that the depth as defined here in terms of infections coincides with the standard meaning of depth -- the graph distance 
 from the root $i$ to a vertex.

%A \emph{witness structure} $S_i$ is a graph on $[t]$ rooted at $i$ where
%\begin{enumerate}
%\item[1.] the root has degree $r$ and the remaining vertices have 
%degree either greater than $r$ or 1. We call the former \emph{internal vertices} and the latter \emph{leaves}; 
%\item[2.] every internal vertex $j$ is associated with $r$ of its incident edges and declares the other endvertices as its children 
%$\mathcal{C}(j)$ and $j$ is their parent.  
%(One would simply choose $r$ of its neighbours, had there been no multiple edges in the graph.)
%The sets satisfy the following property: for any distinct $j,j'$, where $j,j'$ are either internal vertices or the root, we have 
%$j \not \in \mathcal{C}(j')$ and $j' \not \in \mathcal{C}(j)$. In other words, it cannot be the case that $j$ is the parent of $j'$ 
%and $j'$ is simultaneously the parent of $j$; 
%\item[3.] \emph{only} the leaves are all among the set of initially infected vertices and each of them is a child of some internal vertex.  
%\end{enumerate}
%If $S_i$ is not a tree, then there will be vertices with more than one parent. 

Our aim is to bound from above the expected number of witness structures that are rooted at $i$. 
To this end, we will bound this expected value by the expected number of occurrences of a tree which 
is produced from this witness structure through a bounded number of transformations. Informally, during each transformation 
we ``destroy" vertices which belong to cycles in this witness structure. 
Eventually,  having destroyed all such vertices we will obtain a tree whose vertices are equipped with certain valuation functions. 
We finally bound the expected number of occurrences of this tree using Lemma~\ref{fLemma}.

As with trees, we let $\overrightarrow{\mathcal{S}}$ denote an isomorphism class of a witness structure. This can be viewed 
as a directed graph whose vertices $x_0,\ldots, x_N$ are variables taking values in $[t]$, that satisfies 
Conditions {\bf (1)} and {\bf (3)}. We assume that its root is $x_0$.  If $S$ is a witness structure on $[t]$ that is isomorphic to
$\overrightarrow{\mathcal{S}}$, where adjacent vertices are compatible with the directions of the corresponding edges of 
$\overrightarrow{\mathcal{S}}$, then we write $S \in \overrightarrow{\mathcal{S}}$. 
We let $\overrightarrow{\mathcal{S}_i}$ denote the subset of the isomorphism class $\overrightarrow{\mathcal{S}}$, where 
the root is vertex $i$. That is, $x_0=i$. 

Let $X_{\overrightarrow{\mathcal{S}_i}}$ count the
number of copies $S_i \in \overrightarrow{\mathcal{S}_i}$ such that $S_i \subseteq \PA_t(m,\d)$ and $L=\text{leaves}(S_i) \subseteq \mathcal{I}_0$. We have
\begin{equation} \label{mainEq_Gen}
\E\left [X_{\overrightarrow{\mathcal{S}_i}}\right] = \bfrac{1}{\omega t^\g}^{|L|}\sum_{S_i \in \overrightarrow{\mathcal{S}_i}}\Pr\brac{S_i \subseteq \PA_t(m,\d)}.
\end{equation}
Using Corollary~\ref{edgesJointProb} we have  
\begin{equation*}
\sum_{S_i \in \overrightarrow{\mathcal{S}_i}}\Pr\brac{S_i \subseteq \PA_t(m,\d)}
\leq C_2(m,\d,r) \sum_{S_i \in \overrightarrow{\mathcal{S}_i}}\prod_{(a,b) \in E(S_i)}\frac{1}{b^\g a^{1-\g}}  .
\end{equation*}
where $C_2(m,\d,r)$ is some constant that depends only on $m, \d, r$ and $E(S_i)$ denotes the edge set of $S_i$. 

As in the case of trees, we will consider the notion of a \emph{generalised witness structure}, where  
each vertex $x_a$ is associated with a \emph{valuation function} $v_a : [t] \rightarrow \mathbb{R}^+$. 
The valuation functions we consider are as those we considered in the previous section.

Given such a witness structure $\overrightarrow{\mathcal{S}}$, we will define a function 
$f_{\overrightarrow{\mathcal{S}}}: [t] \rightarrow \mathbb{R}^+$, which generalises the weight function of 
a tree that was defined in the previous sub-section. 
When the valuation functions are trivial, then $f_{\overrightarrow{\mathcal{S}}}(i)$
is (up to multiplicative constants) the expected 
number of occurrences of $\overrightarrow{\mathcal{S}}$ rooted at $i$, in the product space of $\PA_t(m,\d)$ and the set of initially
infected vertices. Assume that the vertices of 
$\overrightarrow{\mathcal{S}}$ are $x_0, ,\ldots, x_N$, where $x_0$ is the root. We will be associating the index $j_a$ with the variable $x_a$. 
Also, recall that the edges of $\overrightarrow{\mathcal{S}}$ are directed and 
therefore the edges are \emph{ordered} pairs. 
Letting $j_0=i$, we set 
\begin{equation*} \label{eq:f_Def} 
f_{\overrightarrow{\mathcal{S}}} (i) = v_0(i) \sum_{j_1,\ldots ,j_N} \prod_{a=1}^N v_a (j_a) \prod_{(x_a,x_b) \in E(\overrightarrow{\mathcal{S}_i})} 
\frac{1}{j_b^\g j_a^{1-\g}} \mathbf{1}_{\{ j_a > j_b \}}.
\end{equation*}
It is not hard to see that if $\overrightarrow{\mathcal{S}}$ is a tree, then the above function coincides with the function $f_0(i)$.

Fix a directed isomorphism class $\overrightarrow{\mathcal{S}}$.  
We demonstrate how a sequence of transformations can 
transform $\overrightarrow{\mathcal{S}}$ into a tree isomorphism class $\overrightarrow{\mathcal{T}}$, such that each class in
the sequence is an upper bound (in terms of expectation of witness structures) for the previous. 
Note that by Lemmas~\ref{noPar} and~\ref{NoInfectedOnSmallCycle}, it suffices to consider witness structures of bounded depth where all
initially infected vertices have degree 1. 

Let $x_a$ be a vertex on a cycle such that it has maximum depth (as defined above in terms of the parent-child relation) 
among all vertices on cycles. 
Let $\overrightarrow{\mathcal{T}}(a)$ be the sub-tree rooted at vertex $x_a$.  
We apply Lemma~\ref{fLemma} to $\overrightarrow{\mathcal{T}}(a)$ and obtain 
$$ f_{\overrightarrow{\mathcal{T}}(a)} (j_a) \lesssim {1\over \omega^{\ell (a)}} \frac{(\log j_a )^{\rho (a)}}{j_a^{y_a}}, $$
where $\rho (a)$ and $y_a$ are as in Lemma~\ref{fLemma}. In particular, $y_a = A \gamma + B(1-\gamma)$, where $A, B$ are non-negative 
integers that satisfy $A+B \geq r$.

We are now ready to define the witness structure $T \overrightarrow{\mathcal{S}}$. 
Assume that $x_a$ has $k> 1$ parents $x_{a_1},\ldots, x_{a_k}$ (not necessarily distinct). Also assume that $x_a$ is connected to $x_{a_1},\ldots, x_{a_h}$ through up edges and to $x_{a_{h+1}},\ldots x_{a_k}$ through 
down edges, where $0\leq h \leq k$. Let $\Delta$ now be the index of a parent of the highest depth among $x_{a_1},\ldots, x_{a_h}$, if
$h>0$. To construct $T \overrightarrow {\mathcal{S}}$ 
\begin{enumerate} 
\item[1.] remove $\overrightarrow{\mathcal{T}}(a)$ together with $x_a$;
\item[2.] multiply $v_{a_\Delta} (j_{a_\Delta})$ by
$(\log j_{a_\Delta})^{\rho(a)}j_{a_{\Delta}}^{- y_a}$; 
\item[3.] multiply $v_{a_i} (j_{a_i})$ by  $j_{a_i}^{-\gamma}$, for all $i\not = \Delta$ and $i\leq h$;
\item[4.] multiply $v_{a_i}(j_{a_i})$ by $(\log  j_{a_i})^{\rho(a)+1} / j_{a_i}^{(1-\gamma) \wedge y_a}$, for all $i = h+1,\ldots, k$.
\end{enumerate}
If one of the $x_{a_i}$s is connected to $x_a$ through parallel edges, then the appropriate step from the above is applied once for each 
edge. For the particular case of $x_{a_\Delta}$, Step 2 is applied once for one of the parallel edges, whereas for the others we apply 
Step 3. If the parallel edges are down edges, then we apply Step 4 once for each of them. 

Note that if the valuation functions $v_{a_i}$ which are modified have exponents $e_{a_i}$ satisfying Property (A), then the modifications
incurred by Steps 2-4 preserve this property. Steps 3 and 4 simply remove a child of $x_{a_i}$ and add to the exponent $e_{a_i}$ a $\gamma$ 
or a $y_a \wedge 1-\gamma$, thus preserving Property (A). Step 2 removes a child of $x_{a_\Delta}$ and adds $y_a$ to $e_{a_{\Delta}}$. 
But $y_a =  A \gamma + B(1-\gamma)$, for some non-negative integers $A, B$ that satisfy $A+B \geq r$. Hence, Property (A) is also 
preserved for this exponent. 
% Then
% \begin{enumerate}
% \item[Case 1:] if $x_a$ is connected to them only through down edges, we remove $T(a)$ together with $x_a$. 
% If $y_a$, $\rho (a)$ denotes the exponent of $j$ and $\log j$, respectively, in the upper bound on $f_a (j)$ from Lemma~\ref{fLemma},
%  then we multiply 
% $v_{a_i}(j_{a_i})$ by $(\log  j_{a_i})^{\rho(a)+1} / j_{a_i}^{(1-\gamma) \wedge y_a}$, for all $i = 1,\ldots, k$;
% \item[Case 2:] if $x_a$ is connected to $x_{a_1},\ldots, x_{a_k}$ only through up edges, let $\Delta$ be the index of a parent of the highest
% depth among $x_{a_1},\ldots, x_{a_k}$. Erase $T(a)$ and $x_a$; 
% multiply $v_{a_\Delta} (j_{a_\Delta})$ by $(\log j_{a_\Delta})^{\rho (a)}/j_{a_\Delta}^{y_a}$ and multiply $v_{a_i} (j_{a_i})$
% by  $j_{a_i}^{-\gamma}$, for all $i\not = \Delta$;
% \item[Case 3:] if $x_a$ is connected to $x_{a_1},\ldots, x_{a_h}$ through up edges and to $x_{a_{h+1}},\ldots x_{a_k}$ through 
% down edges, then let $\Delta$ now be the index of a parent of the highest
% depth among $x_{a_1},\ldots, x_{a_h}$. Erase $T(a)$ and $x_a$; multiply $v_{a_\Delta} (j_{a_\Delta})$ by
% $(\log j_{a_\Delta})^{\rho(a)}j_{a_{\Delta}}^{-(k-h)\g -(h-1)(1-\g)- y_a}$, multiply $v_{a_i} (j_{a_i})$
% by  $j_{a_i}^{-\gamma}$, for all $i\not = \Delta$.
%% with $i \leq h$ and for $i>h$ multiply $v_{a_i} (j_{a_i})$ by 
%%$\log j_{a_i}/ j_{a_i}^{1-\gamma}$. 
% \end{enumerate}

%Let $e_a$ and ($(Te)_a$ respectively) the exponent of $j$ in $v_a (j)$ for $x_a \in \overrightarrow{\mathcal{S}_i}$ 
%($x_a \in T \overrightarrow{\mathcal{S}_i}$, resp.). 
Steps 2-4 yield 
\begin{equation} \label{eq:exp_induction} 
\sum_{a' \in T \overrightarrow{\mathcal{S}}} e_{a'} = \sum_{a' \in \overrightarrow{\mathcal{S}} 
\setminus \overrightarrow{\mathcal{T}}(a)} e_{a'} 
+ \begin{cases} k\min \{1-\g, y_a \} & \ \mbox{if $h=0$} \\
(h-1)\g + y_a + (k-h)\min \{1-\g, y_a \} & \ \mbox{if $h>0$}
\end{cases}.
\end{equation}
As we shall see in the proof of the next lemma, Steps 2-4 essentially correspond to a step among a sequence of 
steps that transform $\overrightarrow{\mathcal{S}}$ into a tree. 
In each step, we have the creation of copies of $x_a$, which we denote by $x_{a^{(1)}}, \ldots, x_{a^{(k)}}$, where $x_{a^{(i)}}$ is 
attached to $x_{a_i}$ through an up edge if $i\leq h$ or through an down edge if $i>h$.  
Thereafter, $x_{a^{(\Delta)}}$ as well as $x_{a^{(i)}}$, for $i>h$, each becomes the root of a copy of 
$\overrightarrow{\mathcal{T}}(a)$, 
whereas for the remaining $i$s, the vertices $x_{a^{(i)}}$ become original leaves (cf. Figure~\ref{fig2}). 
We denote the resulting directed graph by $\hat{T} \overrightarrow{\mathcal{S}}$.
Note that this is not the directed graph $T \overrightarrow{\mathcal{S}}$. 
The latter may be thought as coming from $\hat{T} \overrightarrow{\mathcal{S}}$
with the subtrees rooted at each $x_{a^{(i)}}$ 
\emph{contracted} into $x_{a_i}$, multiplying the corresponding valuation functions of $x_{a_i}$ by certain factors, as in Steps 2-4. 
These factors are upper bounds 
on the probability that $x_{a^{(i)}}$ will be infected through the sub-tree that is rooted at it. 
\smallskip

\noindent {\bf Remark} Note also that the depth of $\hat{T} \overrightarrow{\mathcal{S}}$ is equal to the depth of
$\overrightarrow{\mathcal{S}}$. This is the case as all $x_{a^{(i)}}$, for $i>h$, are the roots of a copy of
$\overrightarrow{\mathcal{T}}(a)$ as well as
$x_{a^{(\Delta)}}$. The latter is adjacent to the deepest parent $x_{a_{\Delta}}$ among the $x_{a_i}$s, for $i\leq h$.  

Because of (\ref{mainEq_Gen}), we are interested in the case where the initial witness structure $\overrightarrow{\mathcal{S}}$ has only
trivial valuation functions. In this case, $\sum_{a' \in \overrightarrow{\mathcal{S}}} e_{a'} =0$.   
Assume that each time we apply $T$, we have $1-\g \geq y_a$. 
It follows then from
(\ref{eq:exp_induction}) that during the $j$th transformation the sum of the exponents of the valuation functions increases by 
$\gamma \ell_j$, where $\ell_j$ is the number of leaves that are added during the transition from 
$\hat{T}^{(j-1)} \overrightarrow{\mathcal{S}}$ to $\hat{T}^{(j)} \overrightarrow{\mathcal{S}}$. 
Assume that the process stops after step $j_0$. Thus, $\hat{T}^{(j_0)} \overrightarrow{\mathcal{S}}$ is an $r$-ary tree where all 
its leaves are original and, by the above remark, has depth that is equal to the maximum depth in $\overrightarrow{\mathcal{S}}$. 
If $L_{j_0}$ is the number of leaves of this tree and $\ell ( T^{(j_0)} \overrightarrow{\mathcal{S}} )$ is the number of original leaves 
of $T^{(j_0)} \overrightarrow{\mathcal{S}}$, then 
\begin{equation} \label{eq:leaves} 
L_{j_0}\g = \sum_{a' \in T^{(j_0)} \overrightarrow{\mathcal{S}}} e_{a'} 
+ \ell ( T^{(j_0)} \overrightarrow{\mathcal{S}} )\g. 
\end{equation}

We will use this fact towards the end of our analysis. We now proceed with our basic inductive step which will allow us to bound 
$f_{\overrightarrow{\mathcal{S}}}$ after the application of a sequence of transformations $T$. 

\begin{figure}[htp] 
 \centering
 \includegraphics[scale=0.6]{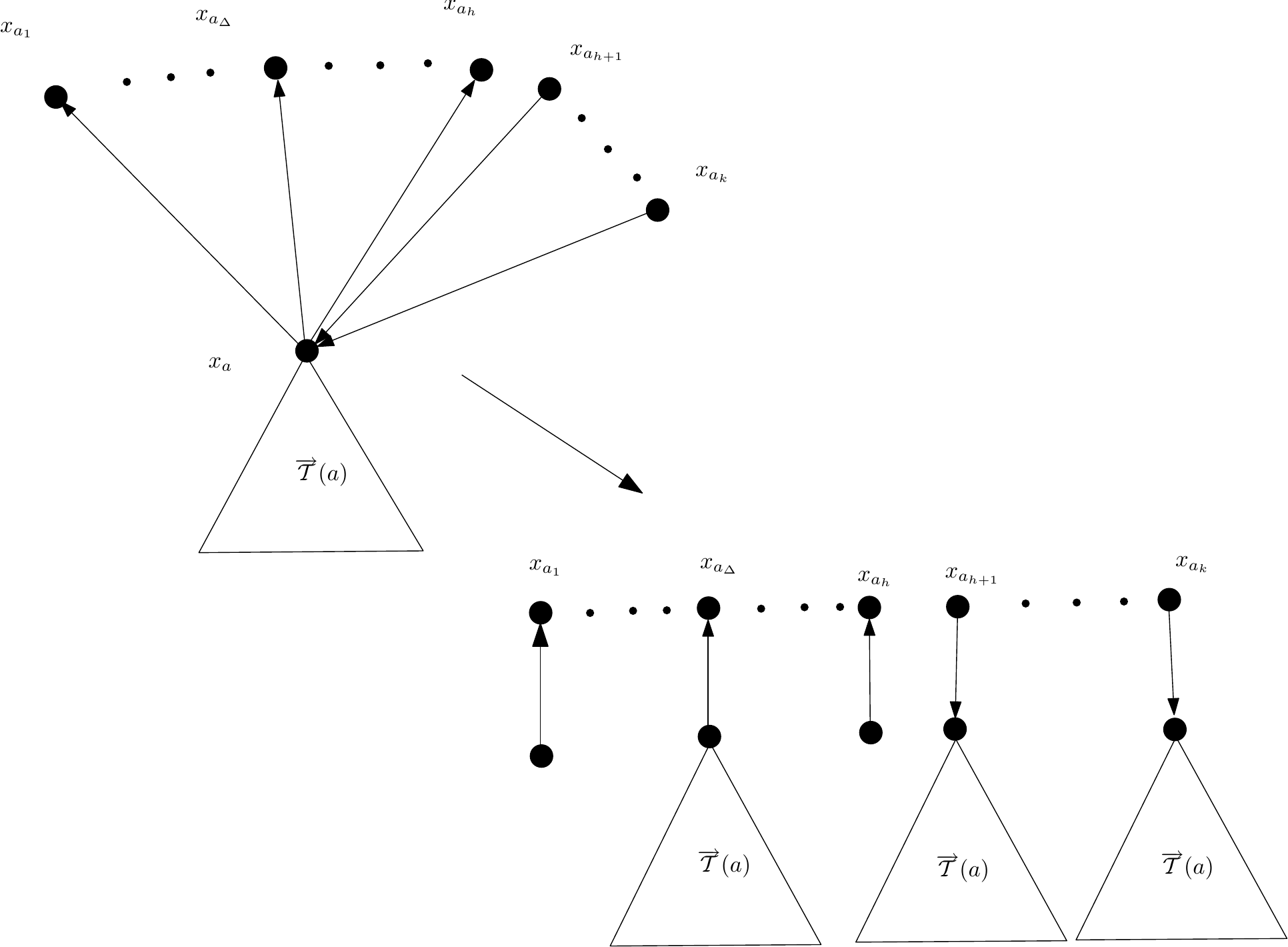}
 \caption{Transformation $\hat{T}$}
 \label{fig2}
\end{figure}

% \begin{figure}[htp] 
%  \centering
%%\vspace{4cm}
%  
%%\hspace{4cm}
%% \includegraphics[scale=0.3, bb=0 0 30 30]{drop_fat.jpg}
%  \caption{Transformation $T$: case 3}
%  \label{fig3}
%  \end{figure}
\begin{lemma} \label{lem:fLemma_Induction} 
Let $\overrightarrow{\mathcal{S}}$ be a witness structure and $x_a$ be a vertex of maximum depth on which we perform 
the above transformation. If $\ell (a)$ denotes the number of original leaves in $\overrightarrow{\mathcal{T}}(a)$, then uniformly for all $i \in [t]$
$$ f_{\overrightarrow{\mathcal{S}}} (i)\lesssim \frac{1}{\omega^{\ell (a)}}~ f_{T \overrightarrow{\mathcal{S}}} (i).$$
\end{lemma}
\begin{proof}
% Assume that the exponent of $j$ is $x$ and that it satisfies $\g<x<C(m,\d, r)$ where $C(m,\d, r)$ is a constant that depends only on $m$,
% $\d$ and $r$. We shall justify 
% this assumption below.

Let $x_a$ be a vertex of $\overrightarrow{\mathcal{S}}$ of maximum depth. Also, 
we denote the set of indices of the vertices of $\overrightarrow{\mathcal{T}}(a)$ by $\mathcal{T}(a)$. 
Let $x_{a_1}, \ldots, x_{a_k}$ be the parents of $x_a$. Denote by $\mathcal{P}(a)$ the set of indices of $\overrightarrow{\mathcal{S}}$ 
not in $\mathcal{T}(a)$. Note that $a_i \in \mathcal{P} (a)$, for $i=1,\ldots, k$.  
Finally assume that the edges $(x_a,x_{a_1}), \ldots, (x_a,x_{a_h})$,
where $0 \leq h \leq k$ are all up edges and $(x_a,x_{a_{h+1}}),\ldots, (x_a,x_{a_{k}})$ are all down edges. 

Now for a set of indices $\mathcal{S}$ we define the function 
$$ f_{\mathcal{S}} (j_a \ :\ a \in \mathcal{S}) = \prod_{a \in \mathcal{S}} v_a (j_a) \prod_{(x_a,x_b) \in E (\mathcal{S} )} 
\frac{1}{j_b^\gamma j_a^{1-\gamma}} \mathbf{1}_{\{ j_a > j_b \}},$$
where $E(\mathcal{S})$ denotes the set of directed edges that is induced by $\mathcal{S}$. 
Using this, we write 
\begin{equation} \label{eq:f_rewrite} 
\begin{split}
&f_{\overrightarrow{\mathcal{S}}}(i)  
= \\
& \sum_{j_{a'} \ : \ a' \in \mathcal{P}(a)\setminus \{a_1,\ldots, a_k \}} \sum_{j_{a_1},\ldots, j_{a_k}} 
f_{\mathcal{P}(a)} (j_{a'}\ : \ a' \in \mathcal{P}(a)) \sum_{j_a >j_{a_1} \vee \cdots \vee j_{a_h}}^{j_{a_{h+1}} \wedge \cdots \wedge
j_{a_k} \wedge t}  f_a (j_a) \prod_{i=1}^h \frac{1}{j_{a_i}^\gamma j_{a}^{1-\gamma}} 
\prod_{i=h+1}^k  \frac{1}{j_a^\gamma j_{a_i}^{1-\gamma}},
\end{split}
\end{equation}
where $f_a$ is the weight function of $\overrightarrow{\mathcal{T}}(a)$. 
We will take an upper bound for each case of the definition of $T$ with the use of Lemma~\ref{fLemma}. 
In particular, using Lemma~\ref{fLemma} we will obtain an upper bound on 
\begin{equation} \label{eq:part_of_f}
\sum_{j_a >j_{a_1} \vee \cdots \vee j_{a_h}}^{j_{a_{h+1}} \wedge \cdots \wedge j_{a_k} \wedge t} 
 f_a (j_a) \prod_{i=1}^h \frac{1}{j_{a_i}^\gamma j_{a}^{1-\gamma}} 
\prod_{i=h+1}^k  \frac{1}{j_{a}^\gamma j_{a_i}^{1-\gamma}}.
\end{equation}
(If $h=k$, then we let $j_{a_{h+1}} \wedge \cdots \wedge j_{a_k} \wedge t=t$.)
% If $(j,j_a)$ is up and $(j,j_b)$ is down then it must be the case that $j_a<j_b$. Without loss of generality, assume 
% $j_1 \leq j_2 \leq \ldots \leq j_h<j<h_{h+1} \leq \ldots j_k$. We do not insist on all being strict inequalities 
% because we allow for parallel edges.

Applying Lemma~\ref{fLemma} to $f_a (j_a)$, we obtain
\begin{equation}
f_a (j_a) \lesssim \frac{1}{\omega^{\ell(a)}} \frac{(\log j_a )^{\rho (a)}}{j_a^{y_a}}, 
\end{equation}
where $y_a$ is as in Lemma~\ref{fLemma}. 

Assume first that $h>0$. Then with $\Delta$ as above, we have
\begin{equation} \label{eq:sum_upper}
\begin{split}
&\sum_{j_a=j_{a_{1}} \vee \cdots \vee j_{a_h}}^{j_{a_{h+1}} \wedge \cdots \wedge j_{a_k} \wedge t} \prod_{i=1}^h 
\frac{1}{j_{a_i}^\gamma j_{a}^{1-\gamma}} 
\prod_{i=h+1}^k  \frac{1}{j_{a}^\gamma j_{a_i}^{1-\gamma}}  f_a (j_a) \lesssim \\
& \frac{1}{\omega^{\ell(a)}} \sum_{j_a=j_{a_{1}} \vee \cdots \vee j_{a_h}}^{j_{a_{h+1}} \wedge \cdots \wedge j_{a_k} \wedge t}
\frac{1}{j_{a_1}^\g j_a^{1-\g}}\ldots\frac{1}{j_{a_h}^\g j_a^{1-\g}}\frac{1}{j_a^\g j_{a_{h+1}}^{1-\g}}
\dots\frac{1}{j_a^\g j_{a_k}^{1-\g}} \frac{(\log j_a )^{\rho (a)}}{j^{y_a}}  \\
&\quad \lesssim  \frac{1}{\omega^{\ell(a)}} \frac{1}{j_{a_1}^\g}\ldots\frac{1}{j_{a_h}^\g}\frac{1}{j_{a_{h+1}}^{1-\g}}\ldots\frac{1}{j_{a_k}^{1-\g}}
\int_{j_a=j_{a_{1}} \vee \cdots \vee j_{a_h}}^{j_{a_{h+1}} \wedge \cdots \wedge j_{a_k} \wedge t}
(\log z )^{\rho (a)} z^{-(k-h)\g -h(1-\g)-y_a} \, \mathrm{d}z  \\ 
& \leq  \frac{1}{\omega^{\ell(a)}} \frac{1}{j_{a_1}^\g}\ldots\frac{1}{j_{a_h}^\g}\frac{1}{j_{a_{h+1}}^{1-\g}}\ldots\frac{1}{j_{a_k}^{1-\g}}
\int_{j_{a_\Delta}}^{t} (\log z )^{\rho (a)}z^{-(k-h)\g -h(1-\g)-y_a} \, \mathrm{d}z.
\end{split}
\end{equation}
But by Lemma~\ref{lem:comb_integral} in Section~\ref{sec:integral} of the Appendix, we have 
\begin{equation*}
\begin{split}
\frac{1}{j_{a_\Delta}^{\g}} \int_{j_{a_\Delta}}^{t} (\log z )^{\rho (a)}z^{-(k-h)\g -h(1-\g)-y_a} &\lesssim
\frac{1}{j_{a_\Delta}^{\g}}~\frac{(\log j_{a_\Delta} )^{\rho (a)}}{j_{a_{\Delta}}^{-1+(k-h)\g +h(1-\g)+y_a}} \\
&= 
\frac{(\log j_{a_\Delta} )^{\rho (a)}}{j_{a_{\Delta}}^{(k-h)\g +(h-1)(1-\g)+y_a}} \leq \frac{(\log j_{a_\Delta} )^{\rho (a)}}{j_{a_{\Delta}}^{y_a}}.
\end{split}
\end{equation*}
Thereby, (\ref{eq:sum_upper}) becomes
\begin{equation} \label{eq:sum_upper_I}
\begin{split}
&\sum_{j_a=j_{a_{1}} \vee \cdots \vee j_{a_h}}^{j_{a_{h+1}} \wedge \cdots \wedge j_{a_k} \wedge t} \prod_{i=1}^h 
\frac{1}{j_a^\gamma j_{a_i}^{1-\gamma}} 
\prod_{i=h+1}^k  \frac{1}{j_{a_i}^\gamma j_{a}^{1-\gamma}}  f_a (j_a) \lesssim  \frac{1}{\omega^{\ell(a)}} 
\frac{1}{j_{a_1}^\g}\ldots\frac{(\log j_{a_\Delta})^{\rho(a)}}{j_{a_\Delta}^{y_a}}
\frac{1}{j_{a_{\Delta+1}}^\g} \ldots \frac{1}{j_{a_h}^\g}\frac{1}{j_{a_{h+1}}^{1-\g}}\ldots\frac{1}{j_{a_k}^{1-\g}}.
\end{split}
\end{equation}

Now, if $h=0$, then (\ref{eq:part_of_f}) yields
\begin{eqnarray}
\sum_{j_a=1}^{j_{a_1} \wedge \cdots \wedge j_{a_k}}\frac{1}{j_a^\g j_{a_1}^{1-\g}}
\dots\frac{1}{j_a^\g j_{a_k}^{1-\g}}\frac{(\log j_a)^{\rho (a)}}{j_a^{y_a}}
&\lesssim& \frac{(\log j_{a_1})^{\rho (a)}}{j_{a_1}^{1-\g}}\ldots\frac{(\log j_{a_k})^{\rho (a)}}{j_{a_k}^{1-\g}}\int_1^{j_{a_k}} z^{-k\g-y_a} \, \mathrm{d}z \nonumber \\
&\lesssim& \left\{ 
  \begin{array}{l l}
     \frac{(\log j_{a_1})^{\rho (a)}}{j_{a_1}^{1-\g}}\ldots\frac{(\log j_{a_k})^{\rho (a)}}{j_{a_k}^{1-\g}}j_{a_k}^{1-k\g-y_a}  & \quad \text{if $1-k\g>y_a$  }\nonumber \\
     \frac{(\log j_{a_1})^{\rho (a)+1}}{j_{a_1}^{1-\g}}\ldots\frac{(\log j_{a_k})^{\rho (a)+1}}{j_{a_k}^{1-\g}}  & \quad \text{if $1-k\g=y_a$  } \nonumber \\
     \frac{(\log j_{a_1})^{\rho (a)+1}}{j_{a_1}^{1-\g}}\ldots\frac{(\log j_{a_k})^{\rho (a)+1}}{j_{a_k}^{1-\g}}& \quad \text{  if $1-k\g<y_a$  }
  \end{array} \right.  \nonumber \\ 
&& \label{785yt8fosvkh0}
\end{eqnarray}
In the first case, we have $y_a < 1- \g$, since $y_a < 1- k\g$. Thus the last factor is
$$\frac{1}{j_{a_k}^{(k-1)\g + y_a}} \leq \frac{1}{j_{a_k}^{y_a}} = \frac{1}{j_{a_k}^{(1-\g) \wedge y_a}} .$$
Hence, in any case (\ref{785yt8fosvkh0}) is bounded by 
\begin{equation} \label{eq:h=0}
\sum_{j_a=1}^{j_{a_1} \wedge \cdots \wedge j_{a_k}}\frac{1}{j_a^\g j_{a_1}^{1-\g}}
\dots\frac{1}{j_a^\g j_{a_k}^{1-\g}}\frac{(\log j_a)^{\rho (a)}}{j_a^{y_a}} \lesssim \prod_{i=1}^{k} 
\frac{(\log j_{a_i})^{\rho (a)+1}}{j_{a_i}^{(1-\g) \wedge y_a}}.
\end{equation}

Setting 
$$ \hat{v}_{a_i} (j_{a_i}) :=
\begin{cases}
\frac{1}{j_{a_i}^{\g }} & \ \mbox{if $i\leq h$ and $i \not =\Delta$},\\
(\log j_{a_i})^{\rho (a)} \cdot \frac{1}{j_{a_i}^{y_a}}, & \ \mbox{if $i =\Delta$ }, \\
 (\log j_{a_i})^{\rho (a) +1} \cdot \frac{1}{j_{a_i}^{(1-\g) \wedge y_a}} & \ \mbox{if $i >h$}, 
\end{cases} 
$$
now (\ref{eq:sum_upper_I}) and (\ref{eq:h=0}) yield
\begin{eqnarray}
\sum_{j_a > j_{a_1} \vee \cdots \vee j_{a_h}}^{j_{a_{h+1}} \wedge \cdots \wedge j_{a_k} \wedge t} 
\prod_{i=1}^h \frac{1}{j_a^\gamma j_{a_i}^{1-\gamma}} \prod_{i=h+1}^k  \frac{1}{j_{a_i}^\gamma j_{a}^{1-\gamma}} f_a (j_a) 
\lesssim&  \frac{1}{\omega^{\ell(a)}} \prod_{i=1}^k \hat{v}_{a_i} (j_{a_i}).
\label{8uy4tuh2}
\end{eqnarray}
So, substituting the bound of (\ref{8uy4tuh2}) into (\ref{eq:f_rewrite}) we obtain
\begin{equation} \label{eq:f_rewrite_upper_I} 
\begin{split}
f_{\overrightarrow{\mathcal{S}}}(i)  
&\lesssim \frac{1}{\omega^{\ell (a)}}~\sum_{j_{a'} \ : \ a' \in \mathcal{P}(a)\setminus \{a_1,\ldots, a_k \}} \sum_{j_{a_1},\ldots, j_{a_k}} 
f_{\mathcal{P}(a)} (j_{a'}\ : \ a' \in \mathcal{P}(a)) \prod_{i=1}^k \hat{v}_{a_i}(j_{a_i}) \\
&= \frac{1}{\omega^{\ell (a)}}
 \sum_{j_{a'} \ : \ a' \in \mathcal{P}(a)\setminus \{a_1,\ldots, a_k \}} \sum_{j_{a_1},\ldots, j_{a_k}}  
\prod_{a' \in \mathcal{P}(a)\setminus \{a_1,\ldots, a_k \}} v_{a'} (j_{a'}) \times \\
& \hspace{3cm} \prod_{i=1}^{k} v_{a_i} (j_{a_i}) \hat{v}_{a_i} (j_{a_i})
\prod_{(x_{a'},x_{b'}) \in E (\mathcal{P}(a))} 
\frac{1}{j_{b'}^\gamma j_{a'}^{1-\gamma}} \mathbf{1}_{\{ j_{a'} > j_{b'} \}} \\ 
&=\frac{1}{\omega^{\ell (a)}} f_{T\overrightarrow{\mathcal{S}_i}}(i).
\end{split}
\end{equation}
Note that the upper bounds in Claim~\ref{clm:sums_cases} imply that 
$\prod_{i=1}^k \hat{v}_{a_i} (j_{a_i})$ is the bound we would get if $x_a$ is replicated $k$ times into 
$x_{a^{(1)}}, \ldots, x_{a^{(k)}}$ and $x_{a^{(i)}}$ is 
attached to $x_{a_i}$ through an up edge if $i\leq h$ or through an down edge if $i>h$.  
Thereafter, $x_{a^{(\Delta)}}$ as well as $x_{a^{(i)}}$, for $i>h$, each becomes the root of a copy of 
$\overrightarrow{\mathcal{T}}(a) $, whereas for the remaining $i$s, the vertices $x_{a^{(i)}}$ become leaves with valuation functions that 
are equal to $1/t^{\gamma}$. Note that the latter is $\omega p_0$ - so essentially these become original leaves. 
\end{proof}

Starting with the original witness structure $\overrightarrow{\mathcal{S}}$, we get a sequence of structures
$T\overrightarrow{\mathcal{S}}, T^{(2)}\overrightarrow{\mathcal{S}},\ldots, T^{(j)}\overrightarrow{\mathcal{S}}$ 
by applying the transformation $T$ in the following way: If $T^{(j-1)}\overrightarrow{\mathcal{S}}$ is a tree, we are done; 
otherwise choose a vertex  $x_a$ of $T^{(j-1)}\overrightarrow{\mathcal{S}}$ such that $x_a$ is on a cycle 
and has maximum depth among such vertices in $T^{(j-1)}\overrightarrow{\mathcal{S}}$. 
Now apply to $x_a$ and its parents the transformation $T$, as appropriate, to
get $T^{(j)}\overrightarrow{\mathcal{S}}$. Note that $T^{(j)}\overrightarrow{\mathcal{S}}$ has at least one less vertex that lies on a
cycle. In general, the number of vertices lying on cycles reduces by one each time we apply the transformation. Hence, there exists a 
$j_0\geq 0$ such that $T^{(j_0)}\overrightarrow{\mathcal{S}}$ is a generalised witness tree. 
Moreover, the depth of this tree is no more than the depth of $\overrightarrow{\mathcal{S}}$.
%In fact, we take the depth of  $\overrightarrow{\mathcal{S}}$ to be equal to $d_0$.

If $x_{a_1}, \ldots, x_{a_{j_0}}$ denote the vertices that were split in each transformation, the repeated application of 
Lemma~\ref{lem:fLemma_Induction} yields
\begin{equation*} 
f_{\overrightarrow{\mathcal{S}}} (i) \lesssim {1\over \omega^{\sum_{j=1}^{j_0}\ell (a_j)}}f_{T^{(j_0)}
\overrightarrow{\mathcal{S}}} (i),
\end{equation*}
where $\ell (a_j)$ is the number of original leaves of $\overrightarrow{\mathcal{T}}(a_j)$ in $T^{(j-1)} \overrightarrow{\mathcal{S}}$.
Observe that $\sum_{j=1}^{j_0}\ell (a_j)$ is the number of leaves in the original witness structure. These leaves are assumed to be original,
and there are at least $r$ of them. 

Since $T^{(j_0)} \overrightarrow{\mathcal{S}}$ is a generalised witness tree, we can apply Lemma~\ref{fLemma} 
and deduce that for some $\rho\geq 0$ and with $\ell$ being the number of original leaves in $T^{(j_0)} \overrightarrow{\mathcal{S}}$
we have 
\begin{equation*} 
f_{\overrightarrow{\mathcal{S}}} (i) \lesssim {1\over \omega^{\ell + \sum_{j=1}^{j_0}\ell (a_j)}} 
\frac{(\log i)^{\rho}}{i^{y}},
\end{equation*}
and either $y$ can be expressed as $A\g + B(1-\g)$ where $A,B$ are non-negative integers such that $A+B\geq r$ and $B>0$ or
$$ y =  \sum_{a' \in T^{(j_0)} \overrightarrow{\mathcal{S}}} e_{a'} 
+ \ell ( T^{(j_0)} \overrightarrow{\mathcal{S}} )\gamma.$$
By  (\ref{eq:leaves}) the latter is equal to $L_{j_0}\gamma$, where $L_{j_0}$ is the number of leaves of 
$\hat{T}^{(j_0)} \overrightarrow{\mathcal{S}}$ . 
But if $\hat{T}^{(j_0)} \overrightarrow{\mathcal{S}}$ is an $r$-ary tree of depth $d_0$ and therefore $\ell_{j_0} \geq d_0$. 
Thus, $y \geq \gamma d_0 >1$. 
Hence (\ref{eq:exp_upper_bound}) also holds in this case, implying that the right-hand side of (\ref{mainEq_Gen}) is $o(1)$. 
Now, if the depth of this tree is less than $d_0$, then by the same principles as in the case of trees we obtain an upper bound
which is $o (t^{1-\g})$. 

This completes the proof of Theorem \ref{Subcritical case}\textbf{(ii)} and \ref{Subcritical case}\textbf{(iii)}.

\section{Critical case}

\begin{proof}[Proof of Theorem \ref{Critical case}\textbf{(i)}]

Let $G$ be a realisation of $\PA_t (m,\delta)$. 
Let $\mathcal{T}(G,d)$ be the set of trees in $G$ which have depth $d$ and for which every internal 
vertex has $r$ children. For a tree $T \in \mathcal{T}(G,d)$ let $A_T$ be the event that all leaves in $T$ are initially infected. 
Note that this is an event on the product space of initial infection, where every vertex is infected independently with probability 
$p$.  Also, note that this is a \emph{non-decreasing} event: if we infect more vertices, then $A_T$ will not stop holding. 

We wish to show $\Pr\brac{\bigcap_{T \in \mathcal{T}(G,d)}A_T^c}>0$. To this end, we apply the FKG inequality  
(see for example Theorem~6.3.2 in~\cite{Janson}):
\begin{eqnarray} \label{8shd8jfjha}
\Pr(\bigcap_{T \in \mathcal{T}(G,d)}A_T^c) &\geq& \prod_{T \in \mathcal{T}(G,d)}\brac{1-\Pr(A_T)} \nonumber \\
&\geq& \exp\brac{-2\sum_{T \in \mathcal{T}(G,d)}\Pr(A_T)},
\end{eqnarray}
where the last inequality follows as $1-x \geq e^{-2x}$ when $x$ is small enough. 
In this case, it will be small enough provided that $t$ is large, since $\Pr(A_T)=p^\ell$ where $\ell$ is the number of leaves in $T$ 
and $p=o(1)$. 

Let $\mathcal{T}(G, d , \ell) \subseteq \mathcal{T}(G, d)$ be those depth-$d$ trees in $G$ with $\ell$ leaves.  We have
\[
\sum_{T \in \mathcal{T}(G,d)}\Pr(A_T)=\sum_{\ell \geq d}\sum_{T \in \mathcal{T}(G, d , \ell)}p^\ell
=\sum_{\ell \geq d}p^\ell\, |\mathcal{T}(G, d , \ell)|
\]

Let $C$ be some large constant, let $\sigma(d,\ell)=\{G \in \PA_t(m,\d): |\mathcal{T}(d , \ell)| \leq C \E [|\mathcal{T}(d , \ell)|] \}$
where $|\mathcal{T}(d , \ell)|$ is the random variable on $\PA_t(m,\d)$ that counts the number of depth-$d$ trees with $\ell$ leaves and
each internal vertex having $r$ children. Let $\sigma(d)=\bigcap_{\ell \geq d}\sigma(d,\ell)$
Then if $G \in \sigma(d)$
\[
\sum_{T \in \mathcal{T}(G,d)}\Pr(A_T) \leq C\sum_{\ell \geq d}p^\ell\, \E [|\mathcal{T}(d , \ell)|]=O(1)
\]
where the last equality follows from~\ref{eq:exp_upper_bound} when $d=d_0$, replacing $\omega$ in $p=p_c/\omega$ (which gave us $o(1)$)
with $1/\lambda$.  

Now 
\[
\Pr(|\mathcal{T}(\PA_t(m,\d), d , \ell)| > C \E [|\mathcal{T}(d , \ell)|]) \leq \frac{1}{C}
\]

hence $\Pr(\PA_t(m,\d) \notin  \sigma(d)) \leq r^{d+1}/C$ since $d\leq \ell \leq r^{d+1}$. Of course, we choose $C>r^{d+1}$. 

Getting back to \eqref{8shd8jfjha}, we see that with probability
at least $1-r^{d+1}/C$,
\[
\Pr(\bigcap_{T \in \mathcal{T}(G,d)}A_T^c | G \in \sigma(d)) \geq \exp\brac{-2\sum_{T \in \mathcal{T}(G,d)}\Pr(A_T)} =\Omega(1).
\]

Consequently, with probability at least $p_1>0$, there is no witness tree of depth $d=d_0$, meaning no infection occurs in this round
or thereafter. 

The same argument applies to witness structures which are not trees. As per above, their expected number of occurrences is 
bounded from above by that of witness trees. 

When $d<d_0$, the results of the previous section show that the expected number of infected vertices in round $d>0$ is $o(t^{1-\g})$.
Hence, the above analysis together with Markov's inequality yields $|\Af|/|\mathcal{I}_0| < 1+\varepsilon$, for $\varepsilon >0$, with 
probability at least $p_1 >0$, for any $t$ large enough.  
%The expected number infected initially is $\Theta(t^{1-\g})$, hence the expected number ever infected is $O(t^{1-\g})$. 
\end{proof}

\begin{proof}[Proof of Theorem \ref{Critical case}\textbf{(ii)}] We wish to show there is a full outbreak. This will happen if, for some 
$k \geq 1$, the first $k$ vertices $[k]$ get
infected, and additionally, no vertex has more than one self-loop. We will show that this happens with some probability
bound away from zero.

Fix a vertex $i$. The argument is along the following lines: The expected
degree of $i$ is about $(t/i)^\g$. Suppose that the actual degree of $i$ is roughly its expected degree. When the infection probability is $p=\lambda/t^\g$ where $\lambda$ is a
constant, then the probability of $i$ getting infected in round $\t=1$ is about 
$\Pr(\text{Bin}((t/i)^\g, \lambda/t^\g)  \geq r) \approx \lambda^r $

For $\d \geq 0$, we can use Lemma~\ref{ColLem1}. Setting $h>0$ to be a sufficiently small constant, we get $\Pr(S_i(t) <
\E[S_i(t)]/K)<1/e^{hi}<1$. Hence, setting $i=1$, we have $S_i(t)=D_i(t)$ and so 
$\Pr(D_1(t) \geq \epsilon t^\g)\geq \epsilon_1$ for some constants $\epsilon, \epsilon_1>0$. For $\d<0$ we apply 
Lemma~\ref{SumConcLem} with $i=1$ to get the same result.

Let $\mathcal{E}_i$  be the event that vertex $i$ has at most one self-loop and let $\mathcal{E}=\bigcap_{i>1}\mathcal{E}_i$.
Let $\mathcal{A}_\epsilon$ be the event $D_1(t) \geq  \epsilon t^\g$. It is clear that 
$\Pr(\mathcal{A}_\epsilon\cap \mathcal{E}) \geq \Pr(\mathcal{A}_\epsilon)\Pr(\mathcal{E})$. 

As per the previous sections, for $i>1$, $\Pr(\mathcal{E}_i^c)=O(1/i^2)$ and so $\liminf_{t \rightarrow \infty}\Pr(\mathcal{E})>0$. 
Therefore, with some probability bounded away from zero, no vertex has more than one self loop, and vertex $1$ is
infected in round $\t=1$. Consequently, all vertices become infected eventually. 
\end{proof}

\section{Conclusions - open questions} This paper studies the evolution of a bootstrap percolation process on random graphs that have
been generated through preferential attachment and generalise the classical Barab\'asi-Albert model. 
For $r < m$, where $2m$ is the average degree, 
we determine a critical function $a_c(t)$ such that when the size $a(t)$ of the initial set ``crosses" $a_c(t)$ 
the evolution of the bootstrap percolation process with activation threshold $r$ changes abruptly from \emph{almost} no evolution to 
full infection. The critical function satisfies $a_c(t)=o(t)$, which implies that a sublinear initial infection leads to full infection. 

Our results are somewhat less tight for $r=2$. It would be interesting to find out whether the sharpness of the threshold that 
we deduced for $r\geq 3$ also holds in this case. Also, the critical window itself for the case $r=2$ has not been explored in the present
work. Furthermore, it would be interesting to determine the number of rounds until 
the complete infection of all vertices in the supercritical case.  

%========================================================================================================================
%========================================================================================================================
%========================================================================================================================
%========================================================================================================================
%========================================================================================================================
%========================================================================================================================

%========================================================================================================================
%========================================================================================================================
%========================================================================================================================
%========================================================================================================================
%========================================================================================================================
%========================================================================================================================

\section{Appendix}
\subsection{Useful facts}
The following are useful facts

For real $x>0$,
\begin{equation}
\G(x+1)=c_x\sqrt{2\pi}e^{-x}x^{x+\frac{1}{2}} \label{useful}
\end{equation}
where $c_x \in [1,e^{\frac{1}{12x}}]$.

Suppose $x \rightarrow \infty$ and $a$ is a constant. Then when $x+a>0,$
\begin{equation}
\frac{\G(x+a)}{\G(x)}= x^a(1+O(1/x)). \label{useful2}
\end{equation}

\subsection{Proofs for sum-of-degree concentrations}
\begin{proof}[\textbf{Proof of Lemma \ref{ColLem1}}]
Assume $h, c_t, A>0$. We shall eventually set $h$ to be a quantity that is $o(1)$.  Let $Z_t=S_i(t)$. 
\[
\Pr\brac{Z_t < A}=\Pr\brac{e^{\frac{-h Z_t}{c_t}}>e^{\frac{-h A}{c_t}}}.
\]
 
$Z_t=Z_{t-1}+Y_t$.  Then $Y_t \succeq X_t \sim \text{Bin}\brac{m,\frac{Z_{t-1}}{mt\brac{2+\d/m}}}$.
\[
\E\left[e^{\frac{-hX_t}{c_t}}\mid Z_{t-1}\right]=\brac{1-p+pe^{\frac{-h}{c_t}}}^m
\]
where $p=\frac{Z_{t-1}}{mt\brac{2+\d/m}}$.

Using $e^{-x} \leq 1- x+x^2$, 
\begin{eqnarray*}
\brac{1-p+pe^{\frac{-h}{c_t}}}^m &\leq& \brac{1-p+p-p\frac{h}{c_t}+p\bfrac{h}{c_t}^2}^m\\
&=& \brac{1-p\frac{h}{c_t}\brac{1-\frac{h}{c_t}}}^m\\
&\leq& \exp \brac{-\frac{mph}{c_t}\brac{1-\frac{h}{c_t}}}\\
&=& \exp \brac{-\frac{hZ_{t-1}}{c_t\brac{2+\d/m}t}\brac{1-\frac{h}{c_t}}}
\end{eqnarray*}

Then
\begin{equation*}
\E\left[e^{\frac{-hZ_{t-1}}{c_t}}e^{\frac{-hY_t}{c_t}}\mid Z_{t-1}\right]\leq \exp \brac{-\frac{hZ_{t-1}}{c_t\brac{2+\d/m}t}\brac{1-\frac{h}{c_t}}-\frac{hZ_{t-1}}{c_t}}.
\end{equation*}

Taking expectations on both sides,
\begin{equation*}
\E\left[\exp \brac{\frac{-hZ_t}{c_t}}\right]\leq \E\left[ \exp \brac{ -\frac{hZ_{t-1}}{c_t} \brac{1+ \frac{1-h/c_t}{(2+\d/m)t} }}\right].
\end{equation*}

Let $c_i=1$ and $c_t=\brac{1+\frac{\g}{t}}c_{t-1}= \brac{1+\frac{1}{(2+\d/m)t}}c_{t-1}$ for $t>i$, and note $c_t \sim \bfrac{t}{i}^\g$. We have,
\begin{eqnarray*}
\E\left[\exp \brac{\frac{-hZ_t}{c_t}}\right]&\leq& \E\left[ \exp \brac{ -\frac{hZ_{t-1}}{c_{t-1}} \frac{1+ \frac{1-h/c_t}{(2+\d/m)t} }{1+\frac{1}{(2+\d/m)t}} }\right] \\
&\leq & \E\left[ \exp \brac{ -\frac{hZ_{t-1}}{c_{t-1}} \brac{1-\frac{h}{(2+\d/m)c_t t}} }\right]. 
\end{eqnarray*}
Iterating,
\begin{eqnarray*}
\E\left[\exp \brac{\frac{-hZ_t}{c_t}}\right]
&\leq & \E\left[ \exp \brac{ -\frac{hZ_{t-1}}{c_{t-1}} \brac{1-\frac{h\g}{c_t t}} }\right]\\
& \leq & \E\left[ \exp \brac{ -\frac{hZ_{t-2}}{c_{t-2}} \brac{1-\frac{h\g}{c_t t}} \brac{1-\frac{h\g}{c_{t-1} (t-1 )}} }\right]\\
&\vdots&\\
&\leq & \E\left[ \exp \brac{ -\frac{hZ_{i}}{c_{i}}\prod_{j=i}^t \brac{1-\frac{h\g}{c_j j}} }\right]\\
&=& \E\left[ \exp \brac{ -2hmi\prod_{j=i}^t \brac{1-\frac{h\g}{c_j j}} }\right].
\end{eqnarray*}
\begin{eqnarray*}
\prod_{j=i}^t \brac{1-\frac{h\g}{c_j j}} &\geq& 1- h\g \sum_{j=i}^t\frac{1}{jc_j}\\
&=&1-O\brac{h\sum_{j=i}^t\frac{1}{j\bfrac{j}{i}^\g}}\\
&=&1-O\brac{hi^\g i^{-\g}}\\
&=&1-O\brac{h}.
\end{eqnarray*}

So
\[
\E\left[\exp \brac{\frac{-hZ_t}{c_t}}\right] \leq \E\left[\exp\brac{-2hmi\brac{1-O\brac{h}}}   \right]=\exp\brac{-2hmi\brac{1-O\brac{h}}}.
\]

Hence using Markov's inequality,
\[
\Pr\brac{e^{\frac{-h Z_t}{c_t}}>e^{\frac{-h A}{c_t}}} \leq \frac{e^{-2hmi\brac{1-O\brac{h}}}}{e^{\frac{-h A}{c_t}}}.
\]

Recalling that $ic_t \sim i\bfrac{t}{i}^\g=t^\g i^{1-\g}$ and $\E[S_i(t)] \geq \beta_1't^\g i^{1-\g}$, choose a sufficiently large constant constant $K$ such that $\E[S_i(t)]/K < \beta'_1 i c_t/\sqrt{K}$ and let $A=\beta'_1 i c_t/\sqrt{K}$. Then,

\begin{eqnarray*}
\Pr\brac{S_i(t) \leq \frac{1}{K}\E[S_i(t)]} &\leq&  \Pr\brac{e^{\frac{-h Z_t}{c_t}}>e^{\frac{-h A}{c_t}}}\\
 &\leq& \exp \brac{-2hmi\brac{1-O\brac{h}}+hi\beta'_1 /\sqrt{K}}\\
&=& \exp \brac{-hi\brac{2m-O(h)-\beta'_1/\sqrt{K}}} \\
&\leq& \exp\brac{-hi},
\end{eqnarray*}
where the last inequality follows if $K>K_0$ where $K_0>0$ is a sufficiently large constant that need only depend on $m,\d$, and if $h$ is small enough.
\end{proof}

\subsection{An integral} \label{sec:integral}

In this section, we prove the following lemma, which has been fairly useful during our calculations. 
\begin{lemma} \label{lem:comb_integral}
Let $k \geq 0$ be an integer, let $\alpha>0$ be a real number and let  
$$I_{k,a}(j):=\int_j^t (\log x)^{k}x^{-1-\alpha}\, \mathrm{d}x.$$  
Then uniformly for $j\geq 1$ we have 
$$ I_{k,a}(j) \lesssim \frac{(\log j)^k}{j^\alpha}.$$
\end{lemma} 
\begin{proof}
Let $v=(\log x)^{k}$, meaning $\frac{dv}{dx}=\frac{k(\log x)^{k-1}}{x}$. Let $\frac{du}{dx}=x^{-1-\alpha}$, meaning 
$u=-\frac{x^{-\alpha}}{\alpha}$. Integration by parts gives
\begin{eqnarray*}
I_{k,a} (j)&=&\left[ - (\log x)^k\frac{ x^{-\alpha}}{\alpha}\right]^t_j + \frac{k}{\alpha}
\int_j^t (\log x)^{k-1}x^{-1-\alpha}\, \mathrm{d}x\\
&\leq& \frac{1}{\alpha }\frac{(\log j)^k}{j^\alpha} + \frac{k}{\alpha}I_{k-1}\\
&\leq& \frac{1}{\alpha }\frac{(\log j)^k}{j^\alpha} + \frac{k}{\alpha}\left[\frac{1}{\alpha }\frac{(\log j)^{k-1}}{j^\alpha} + 
\frac{k-1}{\alpha}I_{k-2}  \right]\\
&=& \frac{1}{\alpha }\frac{(\log j)^k}{j^\alpha} + \frac{k}{\alpha^2}\frac{(\log j)^{k-1}}{j^\alpha} + \frac{k(k-1)}{\alpha^2}I_{k-2}\\
&\leq&\frac{1}{\alpha }\frac{(\log j)^k}{j^\alpha} + \frac{k}{\alpha^2}\frac{(\log j)^{k-1}}{j^\alpha} + 
\frac{k(k-1)}{\alpha^3}\frac{(\log j)^{k-2}}{j^\alpha}+ \\
& & \ldots +  
\frac{k(k-1)\ldots 3}{\alpha^{k-1}}~\frac{(\log j)^2}{j^\alpha}+ \frac{k(k-1)\ldots 2}{\alpha^{k-1}}~I_{1,\alpha}(j).
\end{eqnarray*}
Now 
\[
I_{1,a}(j)=\int_j^t (\log x)x^{-1-\alpha}\, \mathrm{d}x \leq \frac{\log j}{\alpha j^\alpha} + \frac{1}{\alpha^2 j^\alpha}
\]
Thus, we get
\begin{eqnarray*} \label{eq:integral}
I_{k,a} (j) &\leq& \frac{1}{\alpha }\frac{(\log j)^k}{j^\alpha} + \frac{k}{\alpha^2}\frac{(\log j)^{k-1}}{j^\alpha} + 
 \ldots +   \frac{k(k-1)\ldots 2}{\alpha^k}\frac{\log j}{j^\alpha}+ \frac{k!}{\alpha^{k+1}}\frac{1}{j^\alpha}\\
&\leq & \frac{k!}{1 \wedge \alpha^{k+1}}\frac{\log j}{\log j -1}\frac{(\log j)^k}{j^\alpha}.
\end{eqnarray*}
\end{proof}
 
\end{document}